\documentclass[12pt]{amsart}
\usepackage{amssymb}
\usepackage{amsthm}
\usepackage{amsmath}
\usepackage{enumerate}
\usepackage{graphicx}
\usepackage{mathrsfs}
\usepackage{color}

\usepackage[latin1]{inputenc}

\newcommand{\D}{\mathcal D}

\newcommand{\Map}{\mathrm{Map}}
\newcommand{\Out}{\mathrm{Out}}
\newcommand{\Aut}{\mathrm{Aut}}

\theoremstyle{plain}
\newtheorem{theorem}{Theorem}[section]
\newtheorem{theorem*}{Theorem}
\newtheorem{proposition}[theorem*]{Proposition}

\newtheorem{lemma}[theorem]{Lemma}

\newtheorem{prop}[theorem]{Proposition}

\newtheorem{cor}[theorem]{Corollary}

\theoremstyle{definition}
\newtheorem{defi}[theorem]{Definition}

\theoremstyle{remark}
\newtheorem{remark}[theorem]{Remark}

\begin{document}
\title{Spheres and projections for $\mathrm{Out}(F_n)$}
\author{Ursula Hamenst\"adt and Sebastian Hensel}
\date{April 18, 2014}
\thanks{AMS subject classification: 20F65, 57M07.\\  
  Both authors are partially
  supported by the Hausdorff Center Bonn
  and ERC grant 10160104}
\begin{abstract}
  The outer automorphism group $\mathrm{Out}(F_{2g})$ of a free group
  on $2g$ generators naturally contains the mapping class group of a
  punctured genus $g$ surface $S_{g,1}$ as a subgroup. 
  We define a ``subsurface projection'' of the sphere complex 
  of the connected sum of $n$ copies of $S^1\times S^2$ into 
  the arc complex of $S_{g,1}$. Using this, we
  show that $\Map(S_{g,1})$ is a Lipschitz retract of $\mathrm{Out}(F_{2g})$.
  We use another ``subsurface projection'' to give a simple proof 
  of a result of Handel and Mosher \cite{HM10}
  stating that 
  stabilizers of conjugacy classes of free splittings and corank
  $1$ free factors in a free group $F_n$ are Lipschitz retracts of
  $\mathrm{Out}(F_n)$. 
\end{abstract}
\address{\hskip-\parindent
  Ursula Hamenst\"adt\\ Mathematisches Institut der Universit\"at Bonn \\
  Endenicher Allee 60\\
  D-53115 Bonn, Germany}
\email{ursula@math.uni-bonn.de}
\address{\hskip-\parindent
  Sebastian Hensel\\
  The University of Chicago\\
  Department of Mathematics\\
  5734 South University Avenue\\
  Chicago, Illinois 60637-1546\\
  USA}
\email{hensel@math.uchicago.edu}
\maketitle

\section{Introduction}
The \emph{mapping class group} $\Map(S_{g,m})$ of a closed surface 
$S_{g,m}$ of genus $g$ with $m\geq 0$ punctures
is the quotient of the group of 
homeomorphisms of $S_{g,m}$ by the connected component of the
identity. The classical Dehn-Nielsen-Baer
theorem identifies $\Map(S_{g,m})$ with a subgroup of 
the outer automorphism group $\mathrm{Out}(\pi_1(S_{g,m},p))$ of
the fundamental group of $S_{g,m}$ (compare e.g. \cite[Theorem~8.8]{FM11}).
This subgroup consists of all elements which preserve the set of
conjugacy classes of the puncture parallel curves. 

The mapping class group 
acts properly and cocompactly on the
so-called \emph{marking complex} of $S_{g,m}$ 
\cite{MM00} whose vertices are certain finite collections of simple
closed curves. 
In particular, $\Map(S_{g,m})$ is finitely generated, and 
it admits a family of left invariant metrics so that the orbit map for the 
action on the marking complex is a quasi-isometry.

Given an essential subsurface $F$ of $S_{g,m}$, there is 
a simple way to project a simple closed 
curve or a marking of $S_{g,m}$ to a simple closed curve  
or a marking of $F$. This construction determines
a natural coarse Lipschitz retraction
of the marking complex of $S_{g,m}$ onto the 
marking complex of $F$ which is equivariant with respect to 
the action of $\Map(S_{g,m})$ and the stabilizer of $F$ in 
$\Map(S_{g,m})$. A \emph{coarse Lipschitz retraction} of a metric
space $X$ onto a subspace $Y$ is a self-map $p:X\to Y$ which is coarsely
Lipschitz and which restricts to the identity map of $Y$.

Viewing marking complexes as geometric models for the mapping class groups
yields among others a simple construction of 
a Lipschitz retraction of the mapping class group $\Map(S_g)$ onto the
mapping class group $\Map(S_{g-1}^1)$ of a surface with a boundary
component (namely, the subgroup consisting of all elements
which fix a subtorus with connected boundary pointwise up to isotopy
\cite{MM00}). Algebraically, this means that one has a Lipschitz
retraction of the \emph{outer} automorphism group of the fundamental group of a
surface of genus $g\geq 2$ onto the automorphism group of 
the fundamental group of a surface of genus $g-1$.
Here, and in the sequel, we interpret finitely generated groups as
metric spaces by choosing generating sets and equipping them with word norms.

\medskip
The outer automorphism group ${\rm Out}(F_n)$ of the 
free group with $n\geq 2$ generators also admits a simple 
topological model. 
Namely, let $W_n$ be the connected sum of $n$ copies
of $S^1\times S^2$. By a theorem of Laudenbach \cite{L74}, 
${\rm Out}(F_n)$ is a cofinite quotient of the group 
of all isotopy classes of orientation preserving
homeomorphisms of $W_n$. 

Define a \emph{simple sphere system}
in $W_n$ to consist of a collection of essential embedded
spheres which decompose $W_n$ into a union of balls. 
The sphere system graph is the locally finite graph whose 
vertices are simple sphere systems up to isotopy and where two such
sphere systems are connected by an edge if they
can be realized disjointly \cite{HV96,AS09}. One analog of a "subsurface"
is a subset of $W_n$ which is a component $H$ of 
the complement of a collection of disjointly embedded essential spheres. We now
can attempt to define a subsurface projection 
by a surgery
procedure paralleling the construction for surfaces.

However, unlike in the case of simple closed curves on 
a surface, there does not seem to be a canonical way to define
such a projection for sphere systems, and the best we can do is defining 
a projection into the manifold $H_0$ obtained from $H$ 
by filling the boundary spheres with balls. 
This construction results among others in 
a coarsely equivariant projection of ${\rm Out}(F_n)$ into 
the outer automorphism group of the fundamental group of $H_0$,
however we do not obtain a projection  
into the automorphism group
(see \cite{BF12,SS12} for a recent account on such a construction).

If we work with the automorphism group 
${\rm Aut}(F_n)$ of $F_n$ instead then
this problem does not arise (compare \cite{HM10}). 
This observation can be used to 
give a topological proof of a result of Handel and
Mosher \cite{HM10}. For its formulation, note that for
every $n\geq 2$ the subgroup of ${\rm Aut}(F_n)$ of all 
elements which preserve a free splitting 
$F_n=\mathbb{Z}*F_{n-1}$ is naturally isomorphic to 
${\rm Aut}(F_{n-1})$.

\begin{theorem*}\label{aut}
There is a coarsely equivariant 
Lipschitz retraction \[{\rm Aut}(F_n)\to 
{\rm Aut}(F_{n-1}).\]
\end{theorem*}

If we give up on the idea of an \emph{equivariant} 
retraction 
of ${\rm Out}(F_n)$ onto the stabilizer of the conjugacy class of 
a free splitting of $F_n$ then 
we can make consistent choices of basepoints
and use this to 
give a simple topological proof of the following result of Handel and Mosher
\cite{HM10}.

\begin{theorem*}\label{out}
  \begin{enumerate}[i)]
  \item The stabilizer of the conjugacy class of a free splitting $F_n
    = G * H$ is a coarse Lipschitz retract of $\mathrm{Out}(F_n)$.
  \item Let $G < F_n$ be a free factor of corank $1$. Then the
    stabilizer of the conjugacy class of $G$ is a coarse Lipschitz
    retract of 
    $\mathrm{Out}(F_n)$.
  \end{enumerate}
\end{theorem*}

There also are stabilizers of ``subsurfaces'' 
for which we can construct equivariant Lipschitz retractions
in complete analogy to the case of surfaces. Namely,  
for $g\geq 1$ the fundamental group of a 
surface $S_{g,1}$ of genus $g\geq 1$ with 
one puncture is the free group $F_{2g}$. In particular, 
the mapping class group $\Map(S_{g,1})$ 
is a subgroup of ${\rm Out}(F_{2g})$. We define a natural 
subsurface projection of spheres in $W_{2g}$ onto arcs in $S_{g,1}$ 
and use this fact to construct an equivariant 
Lipschitz retraction 
of the sphere system graph of $W_{2g}$ to 
the marking graph of $S_{2g,1}$. This leads to the following

\begin{theorem*}\label{map}
  $\Map(S_{g,1})$ is a coarse Lipschitz retract of $\Out(F_{2g})$.
\end{theorem*}

%More generally,
%for any number $m\geq 1$, the mapping class group
%of a surface $S$ of genus $g\geq 0$ with $m$ punctures
%and fundamental group $F_n$ embeds onto a subgroup of ${\rm
%  Out}(F_n)$. Using straightforward modifications, the methods developed in
%this article show the analog of Theorem~\ref{map} also for such
%surfaces, but we restrict to the case of one boundary component here
%for ease of exposition. 

There is an analog of Theorem \ref{map} for graphs 
which admit cofinite actions of ${\rm Mod}(S_{g,1})$ and 
${\rm Out}(F_{2g})$, respectively. Namely, 
let $\mathcal{AG}(S_{g,1})$ be the arc graph of $S_{g,1}$.
The vertex set of $\mathcal{AG}(S_{g,1})$ is the set of isotopy classes of
essential embedded 
arcs connecting the puncture of $S_{g,1}$ to itself. Two such
vertices are connected by an edge if the corresponding arcs are
disjoint up to homotopy. The mapping class group $\Map(S_{g,1})$ of a
once-punctured surface acts on $\mathcal{AG}(S_{g,1})$.

Define the sphere graph $\mathcal{SG}(W_{2g})$ of
$W_{2g}$ as the graph whose vertex set is the set of isotopy classes
of embedded essential spheres in $W_{2g}$. Two such vertices are
connected by an edge if the corresponding spheres are disjoint up to
homotopy. As for the curve graph of a surface \cite{MM99} or 
the disc graph of a handlebody \cite{MS10}, 
the sphere graph is hyperbolic of infinite diameter \cite{HM11,HiHo12}. 
The tools developed for the proof of Theorem \ref{map}
yield the following analog of a property of the disc graph \cite{MS10,H11}. 
\begin{proposition}\label{prop:arcgraphs}
  There is a $\Map(S_{g,1})$--equivariant embedding of
  the arc graph $\mathcal{AG}(S_{g,1})$ into the sphere graph
  $\mathcal{SG}(W_{2g})$ whose image is
  an equivariant 
  $1$-Lipschitz retract of $\mathcal{SG}(W_{2g})$.
\end{proposition}

\medskip
The article is organized as follows. In Section~\ref{sec:prelim} we
recall the preliminaries on the topological models for $\Out(F_n)$
which we use in the sequel. In Section~\ref{sec:spheres} we prove
Theorem~\ref{aut} and Theorem~\ref{out}. 
Finally, Section~\ref{sec:bundles} contains the proof of the main
Theorem~\ref{map} and Proposition~\ref{prop:arcgraphs}.

\medskip
\textbf{Acknowledgments}: The authors would like to thank the anonymous
referee for numerous suggestions on how to improve the article.
\section{Preliminaries}
\label{sec:prelim}
In this section we set up the basic notation for the rest of the
article and recall some important facts and theorems about $\Out(F_n)$
and its topological models.

Let $F_n$ be the free group of rank $n$. By $\mathrm{Out}(F_n)$ we
denote the outer automorphism group of $F_n$. Explicitly,
$\mathrm{Out}(F_n)$ is the quotient of the automorphism group
$\mathrm{Aut}(F_n)$ by the subgroup of conjugations.

A \emph{free splitting} of $F_n$ consists of two
subgroups $G,H<F_n$ such that $F_n = G*H$.
By this we mean the following: the inclusions of $G$ and $H$ into
$F_n$ induce a natural homomorphism $G*H\to F_n$, where $*$ denotes the
free product of groups. By stating that
$F_n=G*H$ we require that this homomorphism is an isomorphism.
It is possible also to define free splittings using
actions of $F_n$ on trees (see \cite[Section~1.4]{HM10}) but
for our purposes the former point of view is more convenient.
We say that an automorphism $f$ of $F_n$ preserves the free
splitting $F_n=G*H$, if $f$ preserves the groups $G$ and $H$ setwise.

A \emph{corank $1$ free factor} is a subgroup $G$ of $F_n$ of rank
$n-1$ such that there exists a cyclic subgroup $H$ of $F_n$ with $F_n
= G*H$. 
We say that an automorphism $f$ of $F_n$ preserves this
corank $1$ free factor, if $f$ preserves the group $G$. We
emphasize that $f$ is not required to preserve the cyclic group
$H$, and that the group $H$ is not uniquely determined by $G$.

An element $\varphi\in\mathrm{Out}(F_n)$ is said to preserve the
conjugacy class of the free splitting $G*H$ (or corank $1$ free factor
$G$), if there is a representative $f$ of $\varphi$ which
preserves the free splitting $G*H$ (or the corank $1$ free factor
$G$).

\medskip
We will prove our results on the geometry of $\Out(F_n)$ using the topology of the
connected sum $W_n$ of $n$ copies of $S^2\times S^1$ (where $S^k$ denotes
the $k$--sphere). Alternatively, $W_n$ can be obtained by doubling a
handlebody $U_n$ of genus $n$ along its boundary. 
Since $\pi_1(W_n) = F_n$, there
is a natural homomorphism from the group $\mathrm{Diff}^+(W_n)$ of
orientation preserving diffeomorphisms of $W_n$ to $\mathrm{Out}(F_n)$. This
homomorphism factors through the mapping class group
$\Map(W_n)=\mathrm{Diff}^+(W_n)/\mathrm{Diff}_0(W_n)$ of $W_n$, where
$\mathrm{Diff}_0(W_n)$ is the connected component of the identity in
$\mathrm{Diff}^+(W_n)$. In fact, 
Laudenbach \cite[{Th\'{e}or\`{e}me 4.3, Remarque 1)}]{L74} showed that
the following stronger statement is true.
\begin{theorem}\label{thm:laudenbach-unpointed}
  There is a short exact sequence
  $$1\to K\to
  \mathrm{Diffeo}^+(W_n)/\mathrm{Diffeo}_0(W_n)\to\mathrm{Out}(F_n)\to
  1$$
  where $K$ is a finite group, and the map
  $\mathrm{Diffeo}^+(W_n)/\mathrm{Diffeo}_0(W_n)\to\mathrm{Out}(F_n)$ is
  induced by the action on the fundamental group.
\end{theorem}
By \cite[{Th\'{e}or\`{e}me 4.3, part 2)}]{L74}, we can replace
diffeomorphisms by homeomorphisms in 
the definition of the mapping class group of $W_n$. 

An embedded $2$-sphere in $W_n$ is called
\emph{essential}, if it defines a nontrivial element in
$\pi_2(W_n)$. Equivalently, an embedded $2$-sphere is essential if it
does not bound a ball in $W_n$. Throughout
the article we assume that $2$-spheres are smoothly embedded and
essential, unless explicitly stated otherwise.
 
A \emph{sphere system} is a set
$\{\sigma_1,\ldots,\sigma_m\}$ of essential spheres in $W_n$ no two of
which are homotopic. A sphere system is called \emph{simple} if its
complementary components in $W_n$ are simply connected.

\medskip
There are several ways of organizing spheres and sphere systems into
graphs. The \emph{sphere graph} $\mathcal{SG}(W_n)$ is the graph whose
vertex set is the set of isotopy classes of essential spheres in
$W_n$. Two distinct vertices are joined by an edge if the
corresponding isotopy classes of spheres have representatives which
are disjoint. This graph is naturally isomorphic to the free splitting
graph of the free group $F_n$ (compare \cite{AS09}).

\medskip
The \emph{simple sphere system graph} $\mathcal{S}(W_n)$ is the graph whose
vertex set is the set of homotopy classes of simple sphere
systems. Two distinct such vertices are joined by an edge of length $1$ if the 
corresponding sphere systems are disjoint up to homotopy
(compare \cite{Ha95} and \cite{HV96}). We will usually not distinguish
between the homotopy class of a sphere system and the vertex in
$\mathcal{S}(W_n)$ it defines.
We note that if $\Sigma \subset \Sigma'$ are two sphere systems which
are contained in each other, then the corresponding vertices in
$\mathcal{S}(W_n)$ are joined by an edge (one can homotope $\Sigma$
slightly off itself to make it disjoint). In fact, there is a coarse
converse:
\begin{lemma}\label{lem:inclusion-exclusion}
  There is a number $K>0$ with the following property. If
  $\Sigma,\Sigma'$ are disjoint sphere systems, then $\Sigma'$ can be
  obtained from $\Sigma$ by at most $K$ moves, each of which removes a
  sphere or adds a disjoint one. 
\end{lemma}
\begin{proof}
If $\Sigma$ and $\Sigma'$ are disjoint, then they are contained in a
common maximal sphere system. Every maximal sphere system in $W_n$ has
exactly $3n-3$ spheres (compare \cite{Ha95}), and thus $K=3n-3$
satisfies the requirement in the lemma.
\end{proof}

As a consequence of Lemma~\ref{lem:inclusion-exclusion}, the graph in
which edges correspond to  
inclusions (see e.g. \cite{HiHo12}) is quasi-isometric to
$\mathcal{S}(W_n)$, and therefore the 
arguments of this article would also apply to that graph. 

The surgery procedure described in Section~3 of \cite{HV96} shows that
$\mathcal{S}(W_n)$ is connected. Furthermore,
the mapping class group of $W_n$ acts on $\mathcal{S}(W_n)$ properly 
discontinuously and cocompactly (see e.g. the proof of Corollary~4.4
of \cite{HV96} for details on this). 
The finite subgroup $K$ of ${\rm Map}(W_n)$ occurring in the statement of 
Theorem~\ref{thm:laudenbach-unpointed}  acts trivially on isotopy
classes of spheres and hence ${\rm Out}(F_n)$ acts on
$\mathcal{S}(W_n)$ simplicially as well. 
The \v{S}varc-Milnor lemma then applies and yields the following
\begin{lemma}
  The sphere system graph $\mathcal{S}(W_n)$ is quasi-isometric to
  ${\rm Out}(F_n)$. 
\end{lemma}
In particular, we warn the reader that $\mathcal{S}(W_n)$ is not
quasi-isometric to the sphere graph $\mathcal{SG}(W_n)$. 

We also need a variant $W_{n,1}$ of the manifold $W_n$ with a fixed basepoint.  
All isotopies and homeomorphisms of $W_{n,1}$ are required to fix the
basepoint.  In particular, isotopies of spheres in $W_{n,1}$ are not allowed to
move the sphere across the basepoint.

Most of the discussion above is equally valid for $W_{n,1}$.
Since $W_{n,1}$ has a designated basepoint, every homeomorphism of $W_{n,1}$
induces an actual automorphism of $\pi_1(W_{n,1})$. In fact, we have
\begin{theorem}\label{thm:laudenbach-pointed}
  There is a short exact sequence
  $$1\to K\to
  \mathrm{Diffeo}^+(W_{n,1})/\mathrm{Diffeo}_0(W_{n,1})\to\mathrm{Aut}(F_n)\to
  1$$
  where $K$ is a finite group, and the map
  $\mathrm{Diffeo}^+(W_{n,1})/\mathrm{Diffeo}_0(W_{n,1})\to\mathrm{Aut}(F_n)$ is
  induced by the action on the fundamental group.
\end{theorem}
We call a sphere in $W_{n,1}$ essential, if it does not bound a ball
(which may contain the basepoint). In this way, every essential 
sphere in $W_{n,1}$ defines an essential sphere in $W_n$ by 
``forgetting the basepoint''. 
We define $\mathcal{S}(W_{n,1})$ to be the graph of
simple sphere systems in $W_{n,1}$. Edges in $\mathcal{S}(W_{n,1})$
again correspond to disjointness up to homotopy. In this setting, the
\v{S}varc-Milnor lemma implies 
\begin{lemma}
  The sphere system graph $\mathcal{S}(W_{n,1})$ is quasi-isometric to
  ${\rm Aut}(F_n)$.   
\end{lemma}
There is a natural forgetful map $\mathcal{S}(W_{n,1}) \to
\mathcal{S}(W_n)$ which forgets the basepoint and identifies parallel
spheres. This map is equivariant with respect to the actions of
$\Aut(F_n)$ and $\Out(F_n)$ on spheres.
\section{Stabilizers of spheres}
\label{sec:spheres}

The purpose of this section is to give a topological proof of
the following theorem of Handel and Mosher \cite{HM10}.
\begin{theorem}\label{thm:undistorted-handle-mosher}
  \begin{enumerate}[i)]
  \item The stabilizer of the conjugacy class of a free splitting $F_n
    = G * H$ is a Lipschitz retract of $\mathrm{Out}(F_n)$.
  \item Let $G < F_n$ be a free factor of corank $1$. Then the
    stabilizer of the conjugacy class of $G$ is a Lipschitz retract of
    $\mathrm{Out}(F_n)$.
  \end{enumerate}
\end{theorem}
The following lemma describes the stabilizers occurring in
Theorem~\ref{thm:undistorted-handle-mosher} in the topological terms
discussed in Section~\ref{sec:prelim}.   
The statement is an immediate consequence
of Corollary~21 of \cite{HM10} and a standard topological
argument which is for example presented in \cite{AS09}. 
\begin{lemma}\label{lem:factor-stabilizers-are-sphere-stabilizers}
  \begin{enumerate}[i)]
  \item Let $\sigma$ be an essential separating sphere in $W_n$. Then the
    stabilizer of $\sigma$ in $\Map(W_n)$ projects onto the stabilizer of
    the conjugacy class of a free splitting in
    $\mathrm{Out}(F_n)$. Furthermore, every stabilizer of a conjugacy
    class of a free splitting arises in this way.
  \item Let $\sigma$ be a nonseparating sphere in $W_n$. Then the
    stabilizer of $\sigma$ in $\Map(W_n)$ projects onto the stabilizer of
    the conjugacy class of a corank $1$ free factor in
    $\mathrm{Out}(F_n)$. Furthermore, every stabilizer of a conjugacy
    class of a corank $1$ free factor arises in this way. 
  \end{enumerate}
\end{lemma}
To study stabilizers of essential spheres in $W_n$ we use the
following geometric model for stabilizers of spheres.

For an essential sphere $\sigma$, let $\mathcal{S}(W_n,\sigma)$ be the
complete subgraph of $\mathcal{S}(W_n)$ whose vertex set is the set of
homotopy classes of simple sphere systems which are disjoint from
$\sigma$ (but which may contain $\sigma$).
The surgery procedure described in \cite{HV96} shows that the graph
$\mathcal{S}(W_n,\sigma)$ is connected.
The mapping class group of $W_n-\sigma$ acts with finitely many orbits
on spheres in $W_n-\sigma$ (since there are only finitely many
homeomorphism types of complements). Thus, by Laudenbach's theorem,
the stabilizer of $\sigma$ in ${\rm Out}(F_n)$ acts cocompactly on
$\mathcal{S}(W_n,\sigma)$.
Thus the \v{S}varc-Milnor lemma immediately implies the following.
\begin{lemma}\label{lem:qi-model-for-out}
  The graph $\mathcal{S}(W_n,\sigma)$ is equivariantly
  quasi-isometric to the 
  stabilizer of $\sigma$ in ${\rm Out}(F_n)$.
\end{lemma}
Combining Lemma~\ref{lem:factor-stabilizers-are-sphere-stabilizers}
and Lemma~\ref{lem:qi-model-for-out}, 
Theorem~\ref{thm:undistorted-handle-mosher} thus reduces to the
following.
\begin{theorem}\label{thm:sphere-retract}
  The subgraph $\mathcal{S}(W_n,\sigma)$ is a Lipschitz retract of 
  $\mathcal{S}(W_n)$.
\end{theorem}
The construction of this Lipschitz retract has two main
steps. First, we will show a version of Theorem~\ref{thm:sphere-retract} for the
manifold $W_{n,1}$ with a basepoint. Then, in a second step, we
will reduce Theorem~\ref{thm:sphere-retract} to the basepointed case.

\subsection{Stabilizers with basepoint}
\label{sec:basepoint-stab}
In this section we are concerned with the basepointed manifold
$W_{n,1}$. We fix throughout two essential spheres $\sigma^-$ and $\sigma^+$
which bound a region homeomorphic to $S^2\times[0,1]$ containing the
basepoint of $W_{n,1}$. 
We call the sphere system $\sigma^\pm=\{\sigma^+,\sigma^-\}$ a \emph{basepoint
  sphere pair}. 
In particular, when ignoring the basepoint, $\sigma^-$ and $\sigma^+$
are isotopic in $W_n$.

We let $\mathcal{S}(W_{n,1},\sigma^\pm)$ be the
complete subgraph of $\mathcal{S}(W_{n,1})$ whose vertex set is the set of
homotopy classes of simple sphere systems which do not intersect
$\sigma^+$ and $\sigma^-$ (i.e. they contain $\sigma^\pm$ or are disjoint
from $\sigma^+$ and $\sigma^-$). We call such systems
\emph{compatible with $\sigma^\pm$}. 
In this section we prove the following.
\begin{theorem}\label{thm:basept-sphere-retract}
  Let $\sigma^\pm$ be a basepoint sphere pair.
  Then the subgraph $\mathcal{S}(W_{n,1},\sigma^\pm)$ is a
  Lipschitz retract of $\mathcal{S}(W_{n,1})$.
\end{theorem}
As an immediate corollary one then obtains Theorem~1 from the
Introduction, by choosing a basepoint sphere pair $\sigma^\pm$ one of
whose complementary components is homeomorphic to $S^1\times S^2$
minus a ball.

\medskip
The main tool used in the proof of Theorem~\ref{thm:basept-sphere-retract} is a surgery procedure that
makes a given simple sphere system in $W_{n,1}$ disjoint from $\sigma^\pm$.
On the one hand, this surgery procedure is inspired by the construction used in
\cite{HV96} to show that the sphere system complex is contractible. 
On the other hand, it is motivated by the subsurface
projection methods of \cite{MM00}.

\medskip
By definition of a basepoint sphere pair, the connected component
$U^o$ of $W_{n,1}-\sigma^\pm$ which contains the basepoint is 
homeomorphic to $S^2\times(0,1)$. We
call $U^o$ the \emph{open product region associated to $\sigma^\pm$}. We further
define $N = W_{n,1} - U^o$ and call it \emph{the complement of
  $\sigma^\pm$}.
If $\sigma^+$ (or, equivalently, $\sigma^-$) is nonseparating, $N$ has
one connected component, and two otherwise. In any case, the boundary
of $N$ consists of $\sigma^+ \cup \sigma^-$.
We let $U = U^o \cup \sigma^\pm$ be the \emph{(closed) product region}
defined by $\sigma^\pm$.

\medskip
Consider now a simple sphere system $\Sigma$ in $W_{n,1}$. By applying a 
homotopy, we may assume that all intersections between 
$\Sigma$ and $\sigma^\pm$ (viewed as a sphere system) are transverse.
We say that $\Sigma$ and $\sigma^\pm$ \emph{intersect minimally} if 
the number of connected components of $\Sigma \cap \sigma^\pm$ is minimal
among all sphere systems homotopic to $\Sigma$ which intersect $\sigma^\pm$
transversely. 

Every simple sphere system $\Sigma$ can be changed by a homotopy to
intersect $\sigma^\pm$ minimally (for details, compare the discussion of
normal position in \cite{Ha95}). 
Unless stated otherwise, we will assume from now on that all spheres and
sphere systems intersect minimally.
Let $\Sigma' \supset \Sigma$ be a simple sphere system and suppose
that $\Sigma$ intersects $\sigma^\pm$ minimally. Then $\Sigma'$ can be
homotoped relative to $\Sigma$ to intersect $\sigma^\pm$ minimally.
In addition, the isotopy class of $\Sigma$ determines the isotopy
class of the intersection $\Sigma\cap\sigma^\pm$ and the isotopy classes
of the sphere pieces of $\Sigma$ defined below (this uniqueness is
also proved in \cite{Ha95}). 
 
The intersection of the spheres in $\Sigma$ with $N$ is a disjoint
union of properly embedded surfaces $C_1,\ldots, C_m$, possibly with boundary.
Each $C_i$ is a subsurface of a sphere in $\Sigma$, and thus it is 
a bordered sphere. If $\Sigma$ contains spheres disjoint from
$\sigma^\pm$ then some of the $C_i$ may be spheres without boundary
components.
We call the $C_i$ the \textit{sphere pieces in $N$} defined by $\Sigma$.

Similarly, the intersection of $\Sigma$ with $U$ is also a disjoint
union of properly embedded surfaces which we call the \emph{sphere
  pieces in $U$} (Figure~\ref{fig:basepoint-pair}). 
By minimal intersection, each such surface is
either an annulus $A$ joining the two boundary spheres of $U$, or a disk $D$
which separates the basepoint in $U$ from the boundary component that
does not intersect $D$ (see Figure~\ref{fig:basepoint-pair}).
  \begin{figure}[h]
    \centering
    \def\svgwidth{0.5\textwidth}
    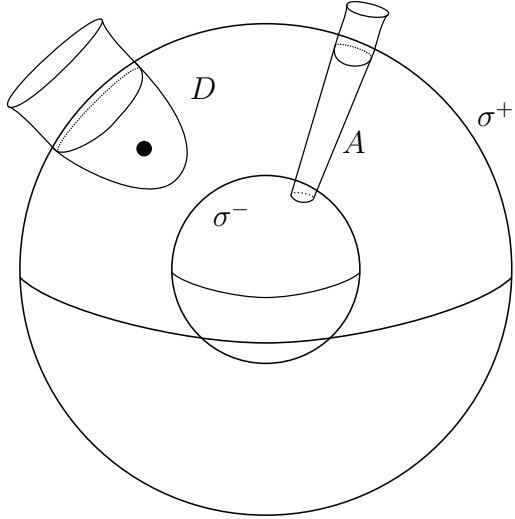  
    \caption{A basepoint pair and two sphere pieces in $U$. The
      basepoint is depicted as the thick black dot. $A$ is contained
      in the inner side of $D$.}
    \label{fig:basepoint-pair}
  \end{figure}
In particular, every connected component of $\Sigma\cap U$ separates
$U$, and exactly one of the complementary components contains the
basepoint. We call this component the \emph{outer side} (motivated by the
intuition that the basepoint is a ``boundary at infinity'').
An intersection circle $\alpha \subset \Sigma \cap \sigma^\pm$ is the
boundary circle of a sphere piece in $U$, and therefore inherits an
outer and inner side on the sphere of $\sigma^\pm$ containing it.

In particular, we will speak about the \emph{inner disk} that a
boundary circle of $C_i$ bounds and mean 
the disk on $\sigma^\pm$ which is disjoint from the outer side of the
corresponding sphere piece in $U$.

\medskip
\begin{lemma}\label{lem:properly-nested}
  Let $D_1,D_2$ be two inner disks for boundary components
  $\alpha_1,\alpha_2$ of possibly different sphere pieces $C_1,C_2$ in
  $N$.

  Then either $D_1$ and $D_2$ are disjoint, or one is properly
  contained in the other.
\end{lemma}
\begin{proof}
  Since $\Sigma$ is embedded, the two circles $\alpha_1,\alpha_2$ are
  disjoint. 
  We may assume that $\alpha_1$ and $\alpha_2$ lie on the same
  boundary component, say $\sigma^+$, as otherwise the claim of the
  lemma is obvious (the inner disks are contained in disjoint spheres).

  Let $C'_1, C'_2$ be the surface pieces in $U$ which have
  $\alpha_1, \alpha_2$ as one of their boundary circles.
  Suppose by contradiction that $D_1, D_2$ are neither disjoint nor
  nested. Then $D_1\cup D_2 = \sigma^+$, and therefore the union of
  the inner sides of $C'_1$ and $C'_2$ is all of $U$. This is
  impossible, since the basepoint in $U$ is contained in neither of
  the two inner sides.
\end{proof}
We abbreviate the conclusion of this lemma by saying that all inner
disks are \emph{properly nested if they intersect}.

\medskip
Let $C$ be one of the sphere pieces of $\Sigma$ in $N$, and let 
$\alpha_1,\ldots,\alpha_k$ be its boundary components on $\partial N$. 
Note that $k$ may be arbitrary large, as opposed to the 
case of surfaces: the intersection of a simple closed curve
with an essential subsurface $Y$ is a union of arcs each of which intersects
the boundary of $Y$ in exactly two points.

Let $\hat{C}$ be the surface obtained from $C$ by gluing the inner
disk $D_i$ along $\partial D_i$ to $\alpha_i$ for each boundary
component $\alpha_i$ of $C$. Since $C$ is a bordered sphere, the  
surface $\hat{C}$ is an immersed sphere in $N$ (which may be inessential). 
We say that $\hat{C}$ is obtained from $C$ by \emph{capping off the
  boundary components.}

\begin{lemma}\label{lem:spheres-admissible-implies-embedded}
  Every sphere obtained by capping off the boundary components of a sphere
  piece is embedded up to homotopy. Furthermore, the 
  spheres obtained by capping off the boundary components of all
  sphere pieces can be embedded disjointly.
\end{lemma}
\begin{proof}
  Let $\mathcal{C}$ be the collection of sphere pieces in $N$ and let
  $\D$ be the set of all inner disks for boundary components of
  $C_i\in\mathcal{C}$. By Lemma~\ref{lem:properly-nested}, $\D$ is a
  collection of properly nested disks. If $\D$ is empty, there is
  nothing to show. 

  Otherwise, say that a disk $D \in \D$ is \emph{innermost} if 
  $D \subset D'$ for every $D' \in \D$ with $D \cap D' \neq
  \emptyset$.
  Since intersecting disks in $\D$ are properly nested, there is at
  least one innermost disk $D_1$ bounded by a curve $\alpha_1$ which
  is the boundary of a sphere piece $C_1$.

  We glue $D_1$ to the corresponding sphere piece $C_1$ and then
  slightly push $D_1$ inside $N$ with a homotopy to obtain a properly 
  embedded bordered sphere $C'_1$ in $N$. Since $D_1$ is innermost, this sphere 
  is disjoint from all sphere pieces $C_k \neq C_1$, and has one
  less boundary component than $C_1$. 

  Now let $\mathcal{C}'$ be the collection of bordered spheres
  obtained from $\mathcal{C}$ by replacing $C_1$ with $C'_1$.
  This is still a collection of disjointly embedded bordered
  spheres. Furthermore, the collection $\mathcal{D}' = \D - \{D_1\}$
  is a collection of properly nested disks, one for each boundary
  circle of a sphere in $\mathcal{C}'$.
  Thus, we can inductively repeat the construction with $\mathcal{C}'$ and
  $\mathcal{D}'$, and the lemma follows.
\end{proof}

We let ${\mathcal P}(\Sigma)$ be the collection
of disjointly embedded spheres obtained by capping off the boundary components of each
sphere piece of $\Sigma$. The set 
${\mathcal P}(\Sigma)$ may contain
inessential spheres and parallel spheres in the same homotopy
class.
We denote by $\pi_{\sigma^\pm}(\Sigma)$ the sphere system obtained as the
union of $\sigma^\pm$ with one representative for each essential
homotopy class of spheres occurring in ${\mathcal P}(\Sigma)$. 
To show that the sphere system obtained in this way from a simple sphere system 
$\Sigma$ is again simple, we require the following topological lemma.

\begin{lemma}\label{lem:pushoff}
  Let $C$ be a sphere piece in $N$ intersecting the boundary of $N$ in
  at least one curve $\alpha$. Let $D \subset \partial N$ be an innermost disk
  with $\partial D = \alpha$. Let $C'$ be the sphere piece obtained by
  gluing $D$ to $C$ and slightly pushing $D$ into $N$ (which might be
  a sphere without boundary components).

  Then every closed
  curve in $N$ which can be homotoped to be disjoint 
  from $C'$ can also be homotoped to be disjoint from $C$.
\end{lemma}
\begin{proof}  
  Pushing the disk $D$ slightly inside of $N$ with a homotopy traces
  out a three-dimensional cylinder $Q$ in $N$. The boundary of $Q$
  consists of two disks (the disk $D$, and the image of $D$ under the
  homotopy) and an annulus $A$ which can be chosen to lie in $C$ (see
  Figure~\ref{fig:pushoff} for an example).

  \begin{figure}[h]
    \centering
    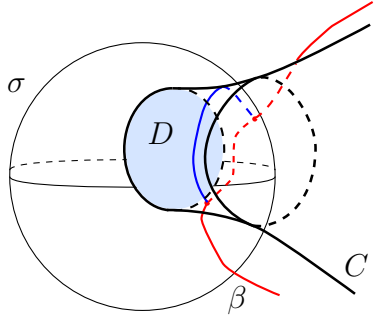  
    \caption{Reducing the number of boundary components of a sphere piece.}
    \label{fig:pushoff}
  \end{figure}

  Suppose that $\beta$ is a closed curve in $N$ which is disjoint from
  $C'$ but not from $C$. Then any intersection point between $\beta$
  and $C$ is contained in the annulus $A$. Up to 
  homotopy, the intersection between $\beta$ and $Q$ is a disjoint
  union of arcs connecting $A$ to itself.
  Since $Q$ is simply connected, each of these arcs can be moved by a
  homotopy relative to its endpoints to be contained entirely in $A$.
  Slightly pushing each of these arcs off $A$ then yields the desired
  homotopy that makes $\beta$ disjoint from $C$.
\end{proof}

\begin{lemma}\label{lem:result-simple-spheres}
 Let $\Sigma$ be a simple sphere system. Then $\pi_{\sigma^\pm}(\Sigma)$ 
  is a simple sphere system.
\end{lemma}
\begin{proof}
  Let $\Sigma$ be a simple sphere system. As $\pi_{\sigma^\pm}(\Sigma)$
  contains $\sigma^\pm$ by construction, it  
  suffices to show that the spheres $S \in \pi_{\sigma^\pm}(\Sigma)$ which are
  distinct from $\sigma^\pm$ decompose $N$ into simply connected regions.
 
  Since the fundamental group of $N$ injects into the fundamental
  group of $W_{n,1}$ and $\Sigma$ is a simple sphere
  system, no essential simple closed curve in $N$ is disjoint from
  $\Sigma\cap N$. In other words, no essential simple closed curve in
  $N$ is disjoint from all sphere pieces defined by $\Sigma$.

  By Lemma~\ref{lem:pushoff}, this property is preserved under capping 
  off one boundary component on a sphere piece. By induction, no
  essential simple closed 
  curve in $N$ is disjoint from all spheres $S \in \mathcal{S}(\Sigma)$.
  Removing inessential spheres and parallel copies of the same sphere
  from $\mathcal{S}(\Sigma)$ does not affect this property.

  This implies that $\pi_{\sigma^\pm}(\Sigma)$ is a simple sphere system
  as claimed.
\end{proof}

\begin{proof}[Proof of Theorem~\ref{thm:basept-sphere-retract}.]
  The image $\pi_{\sigma^\pm}({\mathcal S}(W_{n,1}))$ of the map 
  $\pi_{\sigma^\pm}$ is contained in the subgraph
  $\mathcal{S}(W_{n,1},\sigma^\pm)$, and $\pi_{\sigma^\pm}$ restricts to 
  the identity on the vertex set of $\mathcal{S}(W_{n,1},\sigma^\pm)$. It remains
  to show that it is Lipschitz. 
  
  By using Lemma~\ref{lem:inclusion-exclusion}, it suffices to
  consider the case of two sphere systems $\Sigma \subset \Sigma'$.
  Let $C$ be a sphere piece of $\Sigma$ and let $\sigma$ be the sphere
  in $\Sigma$ containing $C$. Note that the inner disks of the
  boundary circles of $C$ depend only on $\sigma$, not on $\Sigma$.
  This observation implies that
  $\pi_{\sigma^\pm}(\Sigma)\subset\pi_{\sigma^\pm}(\Sigma')$ and in particular
  $\pi_{\sigma^\pm}(\Sigma),\pi_{\sigma^\pm}(\Sigma')$ are
  disjoint. Thus, $\pi_{\sigma^\pm}$ is $K$-Lipschitz, where $K\geq 1$
  is as in Lemma~\ref{lem:inclusion-exclusion}.  
\end{proof}

\subsection{Acquiring a base point}
\label{sec:lifting}
In this section we extend the result from
Section~\ref{sec:basepoint-stab} to the case of the manifold $W_n$
without a base point. To begin, recall the short exact
sequence
\begin{equation}\label{nosplit}
1\to F_n \to \mathrm{Aut}(F_n) \to \mathrm{Out}(F_n) \to 1\end{equation}
since $F_n$ is center-free. 
Similarly, there is a natural map 
$$\mathcal{S}(W_{n,1}) \to \mathcal{S}(W_n)$$
obtained by ``forgetting the base point''. 
This map between graphs is
compatible in with the projection $\Aut(F_n)\to\Out(F_n)$ in the
following sense: by Laudenbach's
Theorems~\ref{thm:laudenbach-unpointed}
and~\ref{thm:laudenbach-pointed}, $\Aut(F_n)$ and $\Out(F_n)$ act on
$\mathcal{S}(W_{n,1}), \mathcal{S}(W_n)$ and these actions are
equivariant with respect to the forgetful map.

By a result of Mosher \cite{Mo96}, since $F_n$ is a nonelementary
word-hyperbolic group, there is a quasi-isometric section\\
\[\Out(F_n)\to\Aut(F_n).\]
Such a section is however by no means unique or canonical.
The graphs $\mathcal{S}(W_{n,1})$ and
$\mathcal{S}(W_n)$ are quasi-isometric to $\Aut(F_n)$ and $\Out(F_n)$
using the orbit map, and therefore Mosher's theorem yields a
quasi-isometric section $s: \mathcal{S}(W_n)\to \mathcal{S}(W_{n,1})$
to the natural forgetful map. 

\medskip
Let $\sigma$ be a sphere in $W_n$, and let $\sigma^\pm$ be a
basepoint sphere pair in $W_{n,1}$ both of whose spheres are homotopic
to $\sigma$ as spheres in $W_n$.

We now describe a procedure which,
intuitively speaking, removes all intersections of spheres in
$W_{n,1}$ with $\sigma^\pm$ which could be removed in $W_n$.  

\medskip
Let $\Sigma$ represent a vertex in $\mathcal{S}(W_{n,1})$. 
We say that an intersection circle $\alpha$ of $\Sigma$ with $\sigma^\pm$ is
\emph{superfluous}, if it bounds a disk $D'\subset\sigma^\pm$ and a
disk $D\subset\Sigma$ such that $D \cup D'$ is inessential in 
$W_n$ (see Figure~\ref{fig:superfluous}).
If furthermore $D$ intersects $\sigma^\pm$ in the single circle
$\partial D$, then we say that $D'$ is a \emph{superfluous surgery
  disk} with \emph{corresponding disk $D$}. The terminology is
well-defined by the lemma below. 

  \begin{figure}[h]
    \centering
    \def\svgwidth{0.7\textwidth}
    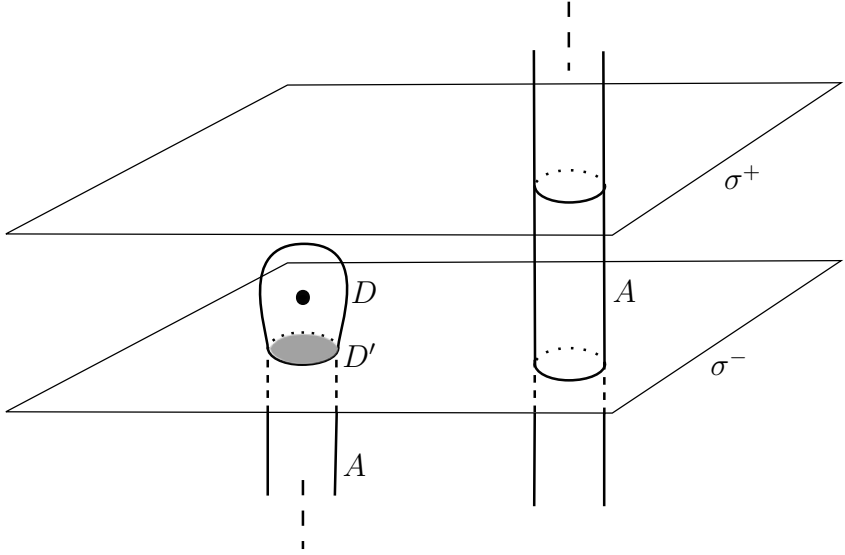  
    \caption{Superfluous circles and disks. The disk $D'$ is a
      superfluous surgery disk. The other intersection circle is also
      superfluous (the union $D\cup A$ of the disk $D$
      and the annulus $A$ yields the desired disk), but
      does not bound a superfluous surgery disk.}
    \label{fig:superfluous}
  \end{figure}

\begin{lemma}\label{lem:nesting-superfluous-disks}
  Suppose $\Sigma$ and $\sigma^\pm$ are in minimal position as sphere
  systems in $W_{n,1}$.
  \begin{enumerate}[i)]
  \item A superfluous intersection circle $\alpha$ of $\Sigma$ with
    $\sigma^\pm$ bounds at most one superfluous surgery disk
    $D'\subset\sigma^\pm$. Furthermore, the corresponding disk
    $D\subset\Sigma$ is also well-defined.

  \item Any two superfluous surgery disks are properly contained in
    each other.

  \item Any intersection circle of $\Sigma$ with $\sigma^\pm$ which is
    contained in a superfluous surgery disk also bounds a superfluous
    surgery disk.

  \item Given a superfluous intersection circle $\alpha$, both
    the superfluous surgery disk and the corresponding disk in
    $\Sigma$ bounded by $\alpha$ only depend on the isotopy class of
    the sphere in $\Sigma$ containing $\alpha$ (not the one of $\Sigma$).
  \end{enumerate}
\end{lemma}
\begin{proof}
  Let $\alpha$ be an intersection circle of $\Sigma$ with
  $\sigma^\pm$. Suppose that $D'\subset\sigma^\pm$ and
  $D\subset\Sigma$ are disks such that $D\cup D'$ is an inessential sphere in
  $W_n$. 

  Then $D\cup D'$ bounds a ball in $B$ in $W_n$. Since we assume that
  $\Sigma$ and $\sigma^\pm$ are in minimal position as sphere systems
  in $W_{n,1}$, they do not bound a ball in $W_{n,1}$ (otherwise, one
  could homotope one of the spheres through this ball to reduce
  intersection). Thus, the ball $B$, seen as a subset of $W_{n,1}$,
  contains the base point and therefore intersects $U$.

  We first show that $D \subset U$ by contradiction. If $D$ is not in
  $U$, then it is contained in $N$. By minimal position of $\Sigma$
  and $\sigma^\pm$, the disk $D$ cannot be homotoped
  relative to its boundary into $\partial N$. Thus, either $N-D$ is connected or
  both components of $N-D$ admit curves which are essential in $N$.
  Hence, for any disk $S'\subset \partial N$, the sphere $D\cup S'$ is
  also either nonseparating in $N$, or admits essential curves in both
  components of $N-(D\cup S')$. Hence, it is essential, contradicting
  the fact that $D$ corresponds to a superfluous surgery disk. 

  Now, for a disk $D$ in $U$ it is easy to see that gluing the disk
  in $\sigma^\pm$ to $D$ which is contained in 
  the outer complementary component yields a sphere which is
  inessential in $W_n$. Gluing on the disk in the inner component
  yields a sphere homotopic to one of the spheres in
  $\sigma^\pm$ (compare Figure~\ref{fig:basepoint-pair}) which is therefore
  essential in $W_n$.
  This shows the desired uniqueness statements in i). 

  The same observation also proves statement ii): the outer components are
  nested for disjoint disks in $U$.
  
  To see iii), suppose that $\hat{\alpha}$ is an intersection circle
  contained in $D'$ as above. Consider the sphere piece $\hat{D}$ in $U$ which
  contains $\hat{\alpha}$ in its boundary. $D$ separates $U$, and
  $\hat{D}$ is contained in the component containing the basepoint by
  definition of superfluous surgery disks. Thus, $\hat{D}$ is itself
  a disk (by minimal position), and the disk bounded by $\hat{\alpha}$
  in $D'$ is a superfluous surgery disk $\hat{D}'$.

  The final claim iv) of the lemma follows, since to detect if a circle or
  disk is superfluous, only information about the sphere containing it
  is needed.
\end{proof}
Next, we describe a canonical way to remove a superfluous surgery disk.
Combining parts ii) and iii) of
Lemma~\ref{lem:nesting-superfluous-disks}, there is an innermost one,
say $D'$. 

By Lemma~\ref{lem:nesting-superfluous-disks}~i) above, there is then
also a unique subdisk $D\subset \Sigma$ such that $D\cup D'$ is inessential in
$W_n$. The \emph{surgery at $D$} is the sphere system obtained
by replacing $D$ by $D'$ (thereby pushing the sphere system across
the basepoint). Note that while this changes the 
isotopy type of the sphere system in $W_{n,1}$, the result still
defines the same sphere system in $W_n$.

\medskip
We now define a ``cleanup'' map $\mathcal{C}:\mathcal{S}(W_{n,1}) \to
\mathcal{S}(W_{n,1})$ in the following way. Let $\Sigma \in \mathcal{S}
(W_{n,1})$ be given. 
Define $\mathcal{C}(\Sigma)$ to be the result of performing surgery at
an innermost superfluous surgery disk, then isotoping the system to be in
minimal position again, and repeating this process
until there are no superfluous surgery disks left. 
Since surgering a superfluous disk does
not change the homotopy type in $W_n$, the result is again a simple
sphere system.

\begin{lemma}\label{lem:r-lipschitz}
  $\mathcal{C}(\Sigma)$ is Lipschitz.
\end{lemma}
\begin{proof}
  By Lemma~\ref{lem:inclusion-exclusion}, it suffices to show that if
  we add or remove a sphere from $\Sigma$, 
  the result of the surgery procedure is disjoint from $\mathcal{C}(\Sigma)$.

  Hence, let $\Sigma \subset \Sigma'$ be given.
  We claim that the spheres in $\mathcal{C}(\Sigma')$ obtained by the
  procedure applied to spheres in $\Sigma$ are exactly the spheres in
  $\mathcal{C}(\Sigma)$. This claim obviously implies the lemma.
  
  The claim now follows since both the normal position
  and the choice of superfluous surgery disks do not 
  depend on the isotopy classes of the full sphere system $\Sigma$ or
  $\Sigma'$, but rather individually on the spheres contained in the
  systems.
\end{proof}

\begin{lemma}\label{lem:disjoint-gets-disjoint}
  If $\Sigma$ can be homotoped to be disjoint from $\sigma$ as a
  sphere system in $W_n$, then $\mathcal{C}(\Sigma)$ is disjoint from
  $\sigma^\pm$. 
\end{lemma}
\begin{proof}
  Put $\Sigma$ and $\sigma^\pm$ in minimal intersection in
  $W_{n,1}$. If these representatives are not in minimal position as
  sphere systems in $W_n$, there is a superfluous intersection
  circle. Therefore, there is also a superfluous surgery disk (by
  considering the innermost superfluous intersection circle).
  Since the procedure defining $\mathcal{C}$
  successively removes all superfluous surgery disks and does not
  change the isotopy class as a sphere system in $W_n$, the result of
  applying $\mathcal{C}$ is disjoint from $\sigma^\pm$. 
\end{proof}

\begin{proof}[Proof of Theorem~\ref{thm:sphere-retract}]
  The desired coarse Lipschitz retraction is constructed by composing several
  maps. Recall that $s:\mathcal{S}(W_n) \to \mathcal{S}(W_{n,1})$ is
  the quasi-isometric section given by Mosher's theorem. In
  particular since $s$ is a section of the basepoint-forgetting-map it
  follows that for every sphere 
  system $\Sigma$ representing a vertex of $\mathcal{S}(W_n)$, the
  sphere system $s(\Sigma)$ is homotopic to $\Sigma$ in $W_n$.

  Let now
  $\sigma$ be an essential sphere in $W_n$, and let as above
  $\sigma^\pm$ be a basepoint sphere pair in $W_{n,1}$ which is homotopic to
  $\sigma$ in $W_n$, and let
  $\mathcal{C}:\mathcal{S}(W_{n,1})\to\mathcal{S}(W_{n,1})$ be the
  corresponding cleanup map defined above. Next, we need the
  Lipschitz retraction $\pi_{\sigma^\pm}$ defined in
  Theorem~\ref{thm:basept-sphere-retract}, and finally the forgetful
  map $f:\mathcal{S}(W_{n,1}) \to \mathcal{S}(W_n)$.

  The desired retraction is now defined as $r = f\circ \pi_{\sigma^\pm}\circ
  \mathcal{C}\circ s$. This map is coarsely Lipschitz, since it is the
  composition of several coarse Lipschitz maps. 
  We claim that its image lies in $\mathcal{S}(W_n,\sigma)$
  and that it restricts to the identity on that set.

  By construction, $\pi_{\sigma^\pm}$ has image in
  $\mathcal{S}(W_{n,1},\sigma^\pm)$. In particular, the image of
  $\pi_{\sigma^\pm}$ consists of sphere systems (in $W_{n,1}$) which
  are disjoint from $\sigma^\pm$. Thus, the image of $r$ consists of
  sphere systems which are disjoint from $\sigma$. Thus shows the
  first claim.

  Now let $\Sigma'$ be any sphere system in $W_n$ which is disjoint
  from $\sigma$. The sphere system $s(\Sigma')$ is homotopic to
  $\Sigma'$ in $W_{n}$ and thus by
  Lemma~\ref{lem:disjoint-gets-disjoint} the image
  $\mathcal{C}\left(s(\Sigma')\right)$ under the clainup map is
  disjoint from $\sigma^\pm$, and still homotopic to $\Sigma'$ in
  $W_n$. Thus, $\pi_{\sigma^\pm}$ fixes this sphere system. Consequently,
  $r(\Sigma') = \Sigma'$, showing the second claim.   
  %The image of
  %$\pi_{\sigma^\pm}$ consists 
  %of sphere systems which are compatible with $\sigma$. Therefore, the image
  %$r(\Sigma)$ of any sphere system has intersection $0$ with
  %$\sigma$. If $\Sigma'$ has intersection $0$ with $\sigma$ in $W_n$, then by
  %Lemma~\ref{lem:disjoint-gets-disjoint} $\mathcal{C}\left(s(\Sigma')\right)$ is
  %disjoint from $\sigma$, and homotopic to $\Sigma'$ in $W_n$. Thus,
  %$\pi_{\sigma^\pm}$ fixes this sphere system. Consequently,
  %$r(\Sigma') = \Sigma'$.   
\end{proof}

\section{Mapping class groups in $\mathrm{Out}(F_n)$}
\label{sec:bundles}

In this section we study the geometry of surface mapping class groups
inside $\mathrm{Out}(F_n)$. Let $S_g^1$ be a surface of genus $g$ with one
boundary component, and let $S_{g,1}$ be the surface obtained by
collapsing the boundary component of $S_g^1$ to a marked point. We
view the marked point as a puncture of the surface, so that the
fundamental group of $S_{g,1}$ is the free group $F_{2g}$ on $2g$ generators.

\begin{remark}
   We believe that our methods can also be used with minor
   modifications to treat the case of more than one boundary
   component, however we did not verify the details. 
\end{remark}
\medskip
A simple closed curve on $S_{g,1}$ which bounds a disk containing the
marked point defines a distinguished conjugacy class in
$\pi_1(S_{g,1})$ called the \emph{cusp class}.

The following analog of the Dehn-Nielsen-Baer theorem for punctured
surfaces is well-known (see e.g. Theorem~8.8 of \cite{FM11}).
\begin{theorem}\label{thm:dnb}
  The homomorphism
  $$\iota:\Map(S_{g,1}) \to \mathrm{Out}(F_{2g})$$
  induced by the action on the fundamental group of $S_{g,1}$ is
  injective. Its image consists of those outer automorphisms which
  preserve the cusp class.
\end{theorem}
In the sequel we identify $\Map(S_{g,1})$ with its image under
$\iota$. The goal of this section is to prove
\begin{theorem}\label{thm:map-in-out}
  $\Map(S_{g,1})$ is a Lipschitz retract of
  $\Out(F_{2g})$. In particular, it is undistorted (i.e. the inclusion
  map is a quasi-isometric embedding).
\end{theorem}
To prove this theorem, we explicitly define a Lipschitz projection map from
$\Out(F_{2g})$ onto the image of $\iota$. This will be
done by a topological procedure in the $3$--manifold $W_{2g}$.
Intuitively speaking, given a simply sphere system $\Sigma$, we will
intersect $\Sigma$ with a nicely embedded copy $C$ of $S_g^1$. The
result $C\cap S_g^1$ is a system of arcs which (coarsely) determines
an element of the surface mapping class group.

There are two main difficulties in this approach. 
First, we need to ensure that $C\cap S_g^1$ is (coarsely) uniquely
defined by the isotopy class of $\Sigma$.
This will be done by defining a normal form for the surface $S_g^1$
with respect to a sphere system. 
Second, we have to show that for a sphere system corresponding to
an element $f$ of the subgroup $\Map(S_{g,1})$, this intersection is
uniformly close to an arc system determined by $f$. 

We now begin with the details of the proof.
\subsection{Geometric models}
\label{sec:geom-models}
Here, we describe the geometric model for the surface mapping class
group $\Map(S_{g,1})$ as a subgroup of $\Out(F_{2g})$ that we will use to prove
Theorem~\ref{thm:map-in-out}. 

\medskip
We begin with the mapping class group of
$S_{g,1}$. A \emph{binding loop system for $S_{g,1}$} is 
a collection of pairwise non-homotopic, essential embedded loops
$\{a_1,\ldots,a_m\}$ based at the marked 
point of $S_{g,1}$ which intersect only at the marked point
and which decompose $S_{g,1}$ into a disjoint union of disks.

Let $\mathcal{BL}(S_{g,1})$ be the
graph whose vertex set is the set of isotopy classes of binding loop
systems. Here, isotopies are required to fix the marked point.
Two such systems are connected by an edge if they intersect
in at most $K$ points different from the base point.
As the mapping class group of $S_{g,1}$ acts with finite quotient on
the set of isotopy classes of binding loop systems (this follows easily
from the change of coordinates principle described in \cite[Chapter~1.3]{FM11}),
we can choose the number $K>0$ such that the following lemma is true.
\begin{lemma}
  The graph $\mathcal{BL}(S_{g,1})$ is connected. The mapping class
  group of $S_{g,1}$ acts on $\mathcal{BL}(S_{g,1})$ with finite
  quotient and finite point stabilizers. 
\end{lemma}
Instead of working with  binding loop systems of $S_{g,1}$ 
directly, we will frequently use \emph{binding arc systems} of $S_g^1$. By
this we mean a collection $A$ of disjointly embedded arcs $\{a_1,\ldots,a_m\}$ 
connecting the boundary component of $S_g^1$ to itself which
decompose $S_g^1$ into simply connected regions.
We will consider such binding arc systems up to isotopy of properly
embedded arcs.
A binding arc system for $S_g^1$ defines a binding loop system for
$S_{g,1}$ by collapsing the boundary component of $S_g^1$ to the
marked point.
Note that if $A_1, A_2$ are two disjoint binding arc systems for
$S_g^1$ then the corresponding binding loop systems for $S_{g,1}$
intersect only at the base point. Therefore, these binding loop
systems are adjacent in $\mathcal{BL}(S_{g,1})$.
The Dehn twist about the boundary component of $S_g^1$ acts trivially
on the isotopy class of any arc system. Thus the action of the
mapping class group of $S_g^1$ on binding arc systems factors through
an action of $\Map(S_{g,1})$.
By the \v{S}varc-Milnor lemma, $\Map(S_{g,1})$ is quasi-isometric to the
graph of binding arc systems via the orbit map.

For the rest of this section, we put $n=2g$.
As before, we use the graph of simple sphere system $\mathcal{S}(W_n)$
as a geometric model for $\mathrm{Out}(F_{2g})$.

\medskip
The inclusion map $\iota:\Map(S_{g,1})\to\Out(F_{2g})$ has a simple
topological description in terms of these models. 
Let $U_{2g} = S_g^1\times[0,1]$ be the trivial oriented interval bundle
over $S_g^1$. $U_{2g}$ is a handlebody of genus $2g$, and we identify
$W_{2g}$ with the three-manifold obtained by doubling $U_{2g}$ along
its boundary. 

Let $a$ be an essential arc on $S_g^1$. Then $a\times[0,1]$ is an
essential disk in $U_{2g}$ which doubles to an essential sphere in $W_n$. A
binding arc system for $S_g^1$ defines a simple sphere system in this
way. 
Similarly, any diffeomorphism $f$ of $S_g^1$ extends 
to a diffeomorphism $I(f)$ of $W_n$ by first extending
$f$ to $U_{2g}$ (by taking the product with the identity on $[0,1]$)
and then doubling to a diffeomorphism of $W_n$. If $f$ and $f'$ are
homotopic as diffeomorphisms of $S_g^1$, then the same is true for
$I(f)$ and $I(f')$. Indeed, the map
$I:\mathrm{Diffeo}(S_g^1)\to\mathrm{Diffeo}(W_n)$ induces the map 
$\iota:\Map(S_{g,1})\to\Out(F_n)$.

\subsection{Minimal position for arcs and curves}
\label{sec:reduc-sphere-syst}
Let $\alpha$ be a closed curve in $W_n$ and let $\Sigma$ be a simple
sphere system. Without loss of generality we always assume that all
such closed curves are embedded, and all intersections with spheres
are transverse. 

We define a minimal position of $\alpha$ with respect to $\Sigma$
as follows. 
Let $W_\Sigma$ be the complement of $\Sigma$ in the sense described
for a single sphere in Section~\ref{sec:spheres} -- that is,
$W_\Sigma$ is a compact (possibly disconnected) three-manifold whose boundary
consists of $2k$ boundary spheres
$\sigma_1^+,\sigma_1^-,\ldots,\sigma_k^+,\sigma_k^-$. The boundary
spheres $\sigma_i^+$ and $\sigma_i^-$ correspond to the two sides of
a sphere $\sigma_i\in \Sigma$.  
The construction of $W_\Sigma$ is analogous to the complements of
single spheres we denoted by $N$ in Section~\ref{sec:spheres}.

If $\alpha$ is not disjoint from $\Sigma$ then
the intersection of $\alpha$ with $W_\Sigma$ is a disjoint union of
arcs connecting the boundary components of $W_\Sigma$. We call these
arcs the \emph{$\Sigma$--arcs of $\alpha$}. An orientation of $\alpha$
induces a cyclic order on the $\Sigma$--arcs of $\alpha$.

We say that \emph{$\alpha$ intersects $\Sigma$ minimally} if no
$\Sigma$-arc of $\alpha$ connects a boundary component of $W_\Sigma$
to itself. Note that this is equivalent to the following statement:
a lift of $\alpha$ into the universal cover of $W_n$ intersects each
lift of each sphere in $\Sigma$ in at most one point.

\begin{remark}
  As the name suggests, $\alpha$ intersects $\Sigma$ minimally if and
  only if the number of intersection points between $\alpha$ and
  $\Sigma$ is minimized among all homotopic representatives of
  $\alpha$ and $\Sigma$. This point of view is not used later in this
  article, and thus we do not give a proof.
\end{remark}
To study minimal position of curves (and later arcs) it is convenient to 
work in the universal cover $\widetilde{W}_n$ of $W_n$.
Let $\widetilde{\Sigma}$ be the full
preimage of $\Sigma$. The dual graph $T_\Sigma$ to $\widetilde{\Sigma}$ is by
definition the graph which has a vertex for each complementary
component of $\widetilde{\Sigma}$, and an edge for each connected
component of $\widetilde{\Sigma}$. It is easy to see that $T_\Sigma$
is in fact a tree. Furthermore, we choose an equivariant retraction $r$
of the manifold $\widetilde{W}_n$ onto $T_\Sigma$. To see that this exists,
consider embedded product neighborhoods $N(\sigma) = \sigma\times(0,1)$ in
$\widetilde{W}_n$ of each sphere $\sigma \in
\widetilde{\Sigma}$. Define a retraction $r$ by mapping each complementary
region of $\cup N(\sigma)$ to the vertex defined by the
corresponding complementary region of $\widetilde{\Sigma}$, and mapping
each region $N(\sigma)$ to the edge in $T_\Sigma$ defined by
$\sigma$ (linearly parametrized by the product coordinate).

If $\alpha$ is any curve or arc, then one can ensure by applying a homotopy 
that every lift $\widetilde{\alpha}$ has the following property: 
each connected component of $\widetilde{\alpha} \cap N(\sigma)$ has the 
form $\{p\}\times(0,1)$ for some $p\in\sigma$. If $\alpha$ was in minimal 
position, then this homotopy furthermore does not change which complementary
components of $\widetilde{\Sigma}$ the lift $\widetilde{\alpha}$ intersects.

\smallskip
We adopt the convention that when working with the tree $T_\Sigma$ we
will always assume that curves have this form. This assumption ensures
that $r(\widetilde{\alpha})$ is a simplicial path in $T_\Sigma$.

\begin{lemma}\label{lem:existence-uniqueness-minimal-intersection-curves}
  \begin{enumerate}[i)]
  \item Every closed curve $\alpha$ in $W_n$ can be modified by a
    homotopy to intersect $\Sigma$ minimally.
  \item Let $\alpha$ and $\alpha'$ be two closed curves which
    are freely homotopic and which intersect $\Sigma$ minimally.  Then
    there is a bijection 
    $f$ between the $\Sigma$--arcs of $\alpha$ and the $\Sigma$--arcs
    of $\alpha'$ such that for each
    $\Sigma$--arc $a$ of $\alpha$ the arc $f(a)$ is homotopic to $a$
    through $\Sigma$-arcs. 

    If orientations of $\alpha$ and
    $\alpha'$ are chosen appropriately, $f$ may be chosen to respect
    the cyclic orders on the $\Sigma$--arcs.
  \end{enumerate}
\end{lemma}
\begin{proof}
  Since $W_\Sigma$ is simply connected, an arc in $W_\Sigma$
  connecting a boundary component to itself can be homotoped through that
  boundary component, reducing the number of intersection points. This
  shows the statement i). 
  
  To see ii), we use the tree $T_\Sigma$.
  Every lift $\widetilde{\alpha}$ of $\alpha$ defines a bi-infinite path
  $r(\widetilde{\alpha})$ in the tree $T_\Sigma$. If $\alpha$
  intersects $\Sigma$ minimally, then $r(\widetilde{\alpha})$ is a
  geodesic in this tree, since it is a path without backtracking. 
  If one modifies $\alpha$ by a homotopy, the chosen 
  lift $\widetilde{\alpha}$ changes by a homotopy as well, and the
  endpoints at infinity of $r(\widetilde{\alpha})$ do not change.
  Since geodesics with given endpoints at
  infinity in a tree are unique, the geodesic $r(\widetilde{\alpha})$
  therefore also does not change.
  Since this geodesic in turn completely determines the intersection 
  pattern and the $\Sigma$-arcs of $\alpha$, the desired uniqueness 
  and part~ii) follows.
\end{proof}

We also need a similar minimal position for arcs with endpoints
sliding on a curve $\delta$. To fix notation, let $\delta$ be an embedded closed
curve in $W_n$. An \emph{arc relative to $\delta$} is an embedded arc $a$
in $W_n$
both of whose endpoints lie on $\delta$. A homotopy of such an arc
will always mean a homotopy though arcs relative to $\delta$. The arc
$a$ is called \emph{essential}, if it is not homotopic to a subset of $\delta$.

\begin{defi}\label{def:minimal-position-arc}
  We say that the arc $a$ is in \emph{minimal position (relative to
    $\delta$)} with respect to a sphere system $\Sigma$ if for one
  (and hence, any) lift $\widetilde{a}$ of $a$ into the universal
  cover of $W_n$ the following hold.
  \begin{enumerate}[i)]
  \item $\widetilde{a}$ intersects each lift $\widetilde{\sigma}$ of a
    sphere in $\Sigma$ in at most one point.
  \item If $\widetilde{\delta}$ is a lift of $\delta$ containing an
    endpoint of $\widetilde{a}$, then it intersects no sphere that
    $\widetilde{a}$ intersects.
  \end{enumerate}
  If the curve $\delta$ is understood, we will simply speak of minimal
  position of $a$.
\end{defi}
Note that part~ii) in particular implies that an endpoint of $a$ is 
not contained in any of the spheres in $\Sigma$.

  \begin{figure}
    \centering
    \def\svgwidth{\textwidth}
    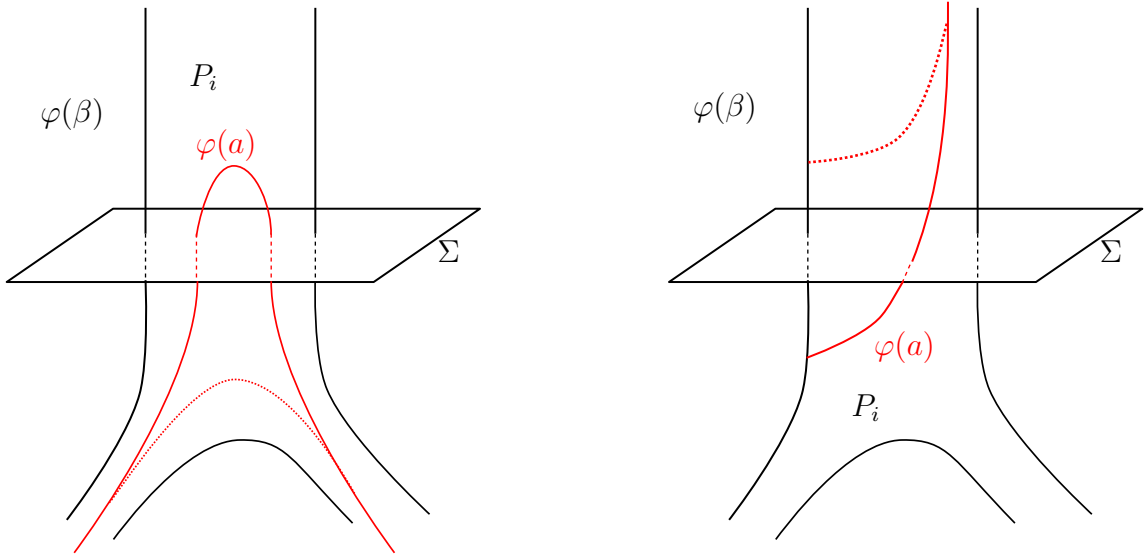  
    \caption{Ways in which arcs can fail to be in minimal position.}
    \label{fig:slide-arc}
  \end{figure}
Figure~\ref{fig:slide-arc} shows the two ways in which an arc
can fail to be in minimal position. 
It is true that $a$ is in minimal
position relative to $\delta$ with respect to $\Sigma$ if and only if
the number of intersections between $a$ and $\Sigma$ is minimized
among all arcs homotopic to $a$. This statement can be shown similarly
to Lemma~\ref{lem:minpos-arcs}, but we do not provide details
since we do not use this characterization of minimal position.

To study minimal position of arcs, we again use the tree $T_\Sigma$ 
defined above. As for curves, we assume that all arcs intersect the
product regions $N(\sigma)$ in straight segments $\{p\}\times(0,1)$.

\begin{lemma}\label{lem:minpos-arcs}
  Let $\Sigma$ be a simple sphere system and $\delta$ be a closed
  curve in minimal position with respect to $\Sigma$. Let $a$ be an
  essential arc relative to $\delta$. Then $a$ can be changed by a
  homotopy to be in minimal position with respect to $\Sigma$.

  On the other hand, suppose that $a$ is in minimal position with
  respect to $\Sigma$ and suppose that it does intersect $\Sigma$.

  Let $\widetilde{a}$ be a lift of $a$ to $\widetilde{W}_n$, and let   
  $\widetilde{U}$ be the complementary component of the full
  preimage $\widetilde{\Sigma}$ of $\Sigma$ which contains the 
  initial point of $\widetilde{a}$.
 
  Then the complementary components of $\widetilde{\Sigma}$ which 
  $\widetilde{a}$ crosses are determined by the homotopy classes 
  of $a$ and $\Sigma$ and the component $\widetilde{U}$. 
\end{lemma}
\begin{proof}
  To see that $a$ can be put in minimal position, we argue in the
  universal cover $\widetilde{W}_n$. The proof is by induction on the
  intersection number between $a$ and $\Sigma$.
  Suppose that $a$ is not in
  minimal position. Then a lift $\widetilde{a}$ violates either
  condition i) or ii) of
  Definition~\ref{def:minimal-position-arc}. Suppose that the first
  condition is violated and 
  $\widetilde{a}$ intersects a lift $\widetilde{\sigma}$ of a sphere
  in $\Sigma$ in at least two points. Then there is a subarc of
  $\widetilde{a}$ which returns to the same side of
  $\widetilde{\sigma}$ (since this sphere is separating in
  $\widetilde{W}_n$). Since each complementary component of
  $\widetilde{\sigma}$ is simply-connected, there is a homotopy of
  $\widetilde{a}$ which pushes this subarc through
  $\widetilde{\sigma}$, thereby reducing the number of intersection
  points.

  Similarly, if condition ii) is violated, then an initial segment of
  $\widetilde{a}$ together with a segment in a lift of $\delta$
  defines an arc that returns to the same side of
  $\widetilde{\sigma}$. Then one slides the initial segment of
  $\widetilde{a}$ to the other side of $\widetilde{\sigma}$, reducing
  the number of intersections.

\medskip
  To show the uniqueness statement, let $\delta_1$ be the lift of
  $\delta$ to $\widetilde{W}_n$ which 
  contains the initial point of $\widetilde{a}$. Similarly, let
  $\delta_2$ be the lift of $\delta$ intersecting $\widetilde{a}$ in
  its endpoint.

  We use the same tree $T_\Sigma$ and retraction $r:\widetilde{W}_n\to
  T_\Sigma$ as in the proof of
  Lemma~\ref{lem:existence-uniqueness-minimal-intersection-curves}. 
  Since $\delta$ is in minimal position with respect to
  $\Sigma$, $r(\delta_1)$ and $r(\delta_2)$ are bi-infinite geodesics
  in $T_\Sigma$. 

  First note that the geodesics $r(\delta_1)$ and
  $r(\delta_2)$ cannot intersect in the tree $T_\Sigma$. Namely, if
  they would intersect, then $\delta_1$ and $\delta_2$ intersect the
  same complementary component of $\widetilde{\Sigma}$ in $\widetilde{W}_n$.
  Then $\widetilde{a}$ could also be homotoped into this complementary
  component, violating the assumption that $a$ intersects $\Sigma$ essentially.

  Since $a$ is in minimal position with respect to $\Sigma$, the retraction
  $r(\widetilde{a})$ is a finite geodesic in $T_\Sigma$ (as it is a
  path without backtracking) or a single point. Furthermore,
  $r(\widetilde{a})$ connects $r(\delta_1)$ to $r(\delta_2)$.

  Therefore $r(\delta_1)$ and $r(\delta_2)$ are disjoint bi-infinite
  geodesics and $r(\widetilde{a})$ is a geodesic connecting the
  $r(\delta_1)$ and $r(\delta_2)$ in $T_\Sigma$. Since geodesics in a 
  tree are unique, the desired statement follows.
\end{proof}

\subsection{Ribbon and minimal position}
\label{sec:ribb-minim-posit}
Let $\varphi_0:S_g^1 \to W_n$ be the embedding of $S_g^1$ into $W_n$
defined by the doubling procedure.
Let $\beta$ be the boundary curve of $S_g^1$. The image
$\varphi_0(\beta)$ is an embedded closed curve in $W_n$ which maps to the
cusp class in $\pi_1(S_{g,1})=\pi_1(W_n)$. 

Next we describe a good position of the surface $\varphi_0(S_g^1)$
with respect to a sphere system.
In fact, we consider the more general case of a surface $\varphi(S_g^1)$, where
$\varphi:S_g^1\to W_n$ is any embedding of $S_g^1$ into $W_n$ which is
homotopic to $\varphi_0$ (note that such an embedding need
not be isotopic to $\varphi_0$). Up to modifying $\varphi$ with a 
small isotopy, we may assume that
$\Sigma$ intersects the surface $\varphi(S_g^1)$ transversely and we
will always do so.
Then the preimage $\varphi^{-1}(\Sigma)$ is a one-dimensional
submanifold of $S_g^1$, and hence it is 
a disjoint union of simple closed curves and properly embedded arcs.

\begin{defi}\label{def:ribbon-position}
  We say that $\varphi$ is in \emph{ribbon position with respect to
    $\Sigma$} if each component of $\varphi^{-1}(\Sigma)$ is a
  properly embedded arc. It is said to be in \emph{minimal position}
  if in addition $\varphi(\beta)$ is in minimal position with respect
  to $\Sigma$. In either case, we call the preimage
  $\varphi^{-1}(\Sigma)$ the \emph{arc system induced by $\Sigma$ and
    $\varphi$}.
\end{defi}

Note that a priori the
homotopy class of the arc system induced by $\Sigma$ and $\varphi$ need
not be determined by the isotopy class of $\Sigma$ even if $\varphi$
is in minimal position with respect to $\Sigma$. We will address this
problem below.

First, we show that $\varphi$ may always be put in minimal
position.
For this we use an inductive
method which is described in the next lemma. In the proof, we need
the following observation which also motivates the terminology ``ribbon
position''.

\medskip
Fix a simple sphere system $\Sigma$ and suppose that $\varphi$ is in
ribbon position with respect to the sphere system $\Sigma$. 
We will develop a convenient combinatorial way to describe the
components of $W_\Sigma \cap \varphi(S_g^1)$.

The intersection of $\varphi(S_g^1)$ with $W_\Sigma$ is a
union of surfaces $P_1,\ldots,P_k$. We call the $P_i$ the
\emph{polygonal disks defined by $\varphi$ relative to $\Sigma$}.
As $\Sigma$ is a simple sphere system, the arc system
$\varphi^{-1}(\Sigma)$ on $S_g^1$ is binding and hence
each of the surfaces $P_i$ is a disk whose
boundary is not completely contained in a boundary component of
$W_\Sigma$ (for the definition of $W_\Sigma$ compare the first
paragraph of Section~\ref{sec:reduc-sphere-syst}).

The disk $P_i$ is already determined by a spine embedded in it, as we
will explain below. This point of view will allow us to modify the
$P_i$ (and therefore the map 
$\varphi$) as if they were one-dimensional objects. Since $W_\Sigma$ is
three-dimensional, this will give us the desired freedom to put $\varphi$ in
a particularly convenient position. 

\medskip
Pick one polygonal disk, say $P_i$, and consider its boundary curve
$\delta_i$. We can write this curve in the form
$$\delta_i = a_1*b_1*\dots*a_r*b_r$$
where each $a_i$ is an arc contained in one of the boundary spheres
of $W_\Sigma$, and each $b_i \subset \varphi(\beta)$ is a properly
embedded arc in $W_\Sigma$ (compare Figure~\ref{fig:embedded-graph}).
 \begin{figure}[h]
    \centering
    \def\svgwidth{0.7\textwidth}
    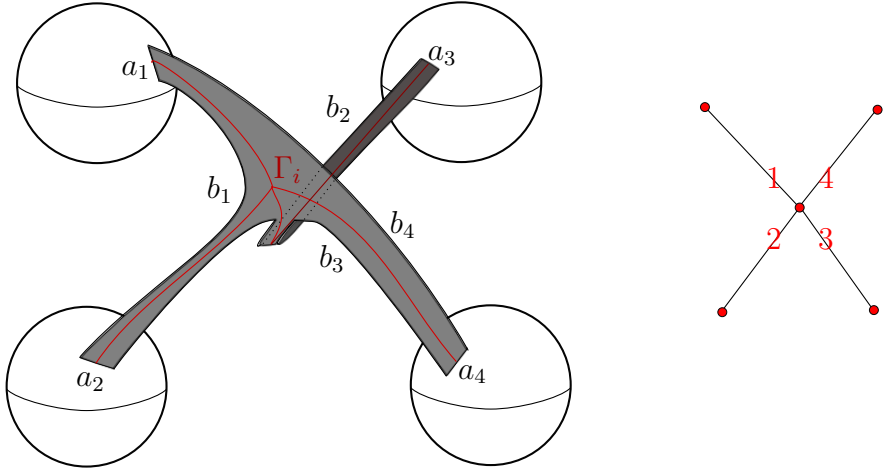  
    \caption{On the left: A polygonal disk $P_i$ (light gray on top,
      dark gray on the bottom) with its embedded copy of
      $\Gamma_i$. On the right: the corresponding graph
      $\Gamma_i$ with the cyclic order (ribbon structure) at the central vertex}
    \label{fig:embedded-graph}
  \end{figure}

We fix the set of arcs $\{a_i,1\leq i\leq r\}$ for the moment and equip each
$a_i$ with an orientation. We call this set of oriented arcs 
the \emph{boundary arcs} of the disk $P_i$. We want to describe the way 
that $P_i$ joins the boundary arcs in $W_\Sigma$ in a combinatorial way.

To this end, let $\Gamma_i \subset P_i$ be an embedded graph in $P_i$ defined in
the following way. The graph $\Gamma_i$ has one distinguished vertex
$v_0$ (called the \emph{interior vertex}) contained in the interior of
$P_i$ and one vertex $v_r$ 
contained in each boundary arc $a_r$ (called \emph{boundary vertices}). 
Each vertex $v_r$ $(r\geq 1)$ is connected by an edge to the vertex $v_0$. 
The oriented surface $P_i$ determines a 
\emph{ribbon structure} on $\Gamma_i$. Recall that a ribbon structure on 
a graph is a cyclic order of the half-edges at each vertex.

A ribbon graph defines a surface by replacing each vertex by a small disk, 
each edge by a small rectangle, and gluing these rectangles to the disks
given by the cyclic order on the vertices. However, in our case
we cannot yet reconstruct $P_i$ (even up to homotopy) from the ribbon graph $\Gamma_i$
for the following reason: there are up to homotopy two possibilities how to glue
in a band between, say, a boundary arc and a disk (compare Figure~\ref{fig:decorated}).
\begin{figure}[h]
  \centering
  \def\svgwidth{0.7\textwidth}
  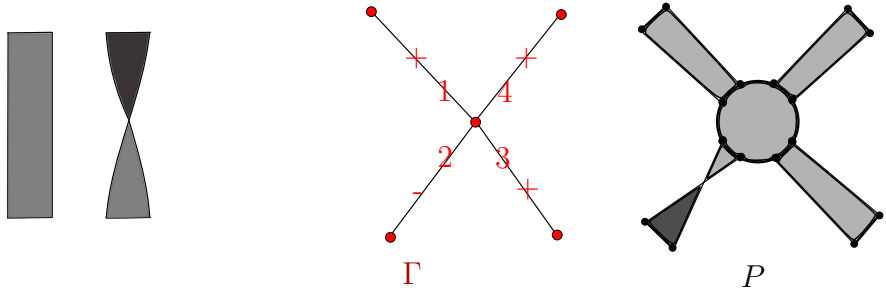  
  \caption{On the left: The two ways of gluing a band between two
    intervals. On the right: An example of a decorated ribbon graph and the
    corresponding surface}
  \label{fig:decorated}
\end{figure}

To avoid this issue, we use the following graphs to model the
polygonal disks $P_i$.
\begin{defi}
  A \emph{decorated ribbon graph} (for the sphere system $\Sigma$ and
  the boundary arcs $\{a_i\}$)
  is an embedded graph $\Gamma$ in $W_\Sigma$ satisfying
  \begin{enumerate}[i)]
  \item $\Gamma$ is a tree which has a valence-$1$ boundary
    vertex contained in each of the $a_i$.
  \item At every interior vertex of $\Gamma$, the adjacent half-edges
    are cyclically ordered (ribbon structure)
  \item Every edge of $\Gamma$ is labeled by a sign $+$ or $-$ (twist datum).
  \end{enumerate}
\end{defi}

\medskip
The \emph{surface associated to a decorated ribbon graph $\Gamma$}
is defined in the following way. Put small embedded 
oriented disks $D_v$ at the interior vertices $v$ of $\Gamma$ which
contain the star of $v$ in $\Gamma$ and such
that the cyclic order of edges at $v$ agrees with the positive
orientation given by $D_v$.

Suppose an edge $e$ connects an interior vertex $v$ to a boundary vertex
corresponding to a boundary arc $a_r$.
We then connect $a_r$ to the disk $D_v$ with a \emph{band} $B_r$, i.e.
an embedded product of two intervals $[0,1]\times[0,1]$ in $W_n$, as follows. 
One of the sides of $B_r$ is the arc
$a_r$, and the opposite side $a'_r$ is contained in $\partial D_v$. We call
these sides the \emph{horizontal sides}. Correspondingly, the
\emph{vertical sides} are properly embedded arcs in $W_n$.
The orientation of $\partial D_v$ 
determines a left and right endpoint
of each of the $a'_r$.

If the edge $e$ is decorated with a $+$, we match
the left endpoint of $a_r$ with the left endpoint of the interval on
$\partial D$, otherwise we pair the left with the right endpoint
(compare Figure~\ref{fig:decorated}).

Similarly, we glue bands between disks corresponding to two different interior
vertices according to the sign on the connecting edge.

\medskip
We will also need to modify decorated ribbon graphs. On the one hand, we will consider
\emph{homotopies of decorated ribbon graphs} which do not change the
combinatorics of the graph.

The other modifications are \emph{split} and \emph{collapse
  moves}. Namely, let $\Gamma$ be a decorated ribbon graph.
A collapse move requires two interior vertices $v_1,v_2$ which are
joined by an edge $e$. The decorated ribbon graph $\Gamma'$ obtained
by collapsing $e$ is the following: 
\begin{enumerate}[i)]
\item The underlying graph of $\Gamma'$ is obtained from $\Gamma$ by collapsing $e$
  to a vertex $\hat{v}$. In particular, every vertex of $\Gamma$ which
  is not equal to $v_1$ or $v_2$ is  
  also a vertex of $\Gamma'$. Every edge of $\Gamma$ except for $e$
  corresponds to a single 
  edge in $\Gamma'$. 
\item The ribbon structure at each vertex different from $\hat{v}$ is
  the same as in $\Gamma$. 
\item The ribbon structure at $\hat{v}$ is obtained in the following
  way. The edges at $\hat{v}$ are all  
  the edges in $\Gamma$ connected to either $v_1$ or $v_2$ (except for $e$).
  Let $e_1,\ldots,e_r,e$ be the cyclic order of edges at $v_1$, and
  $e,e'_1,\ldots,e'_s$ the cyclic 
  order of edges at $v_2$. If the edge $e$ was labeled by a $+$, then the cyclic
  order at $\hat{v}$ is 
  $$e_1,\ldots,e_r,e'_1,\ldots,e'_s$$
  and otherwise it is
  $$e_1,\ldots,e_r,e'_s,\ldots,e'_1$$
\item The twist data at all edges are as in $\Gamma$.
\end{enumerate}
A split move is an inverse to a collapse move. A split move is
defined by splitting the link of an interior vertex into two groups,
each of which is connected 
in the cyclic order. Then an additional edge with sign $+$ is
generated, such that the two groups are contained in the links of two different
vertices. 

\begin{lemma}\label{lem:ribbon-control}
  Fix a simple sphere system $\Sigma$ and boundary arcs $a_i$ in
  $\partial W_\Sigma$. 
  \begin{enumerate}[i)]
  \item Let $P_i$ be any polygonal disk, and $\Gamma_i$ the
    corresponding embedded ribbon graph. Then $\Gamma_i$ admits a
    twisting datum such that $P_i$ is homotopic,
    relative to its boundary, to the surface defined by $\Gamma_i$.
  \item Let $\Gamma$ and $\Gamma'$ be two decorated ribbon graphs
    which are homotopic (as decorated ribbon graphs). Then the
    surfaces defined by $\Gamma$ and $\Gamma'$ are homotopic relative
    to the boundary arcs $a_i$.
  \item Let $\Gamma'$ be obtained from $\Gamma$ by a collapse or 
    split move. Then the surfaces defined by $\Gamma$ and $\Gamma'$ 
    are homotopic relative to their horizontal sides in $\Sigma$.
  \end{enumerate}
\end{lemma}
\begin{proof}
\medskip
  The ribbon graph $\Gamma_i$ contained in $P_i$ has a
  single interior vertex. Therefore, part i) simply follows from the
  fact that there are up to homotopy only two ways to join the central
  disk to each of the boundary arcs. 

  Part ii) simply follows by extending
  a homotopy of a decorated ribbon graph $\Gamma$ to a homotopy of a small 
  regular neighborhood. Since the surface defined by $\Gamma$ may be assumed
  to lie in such a neighborhood, and is uniquely determined (up to homotopy) 
  by the decorated ribbon structure, the claim follows. Part iii) follows
  directly from the definitions of collapse and split moves.
\end{proof}

\begin{lemma}\label{lem:keeping-minimal}
  Let $\Sigma$ be a simple sphere system.
  Suppose that $\varphi$ is in minimal position with respect to
  $\Sigma$. Let $\sigma'$ be an embedded sphere
  disjoint from $\Sigma$ and let $\Sigma'$ be a simple sphere system 
  obtained from $\Sigma$ by either adding $\sigma'$, or by removing 
  one sphere $\sigma\in\Sigma$.

  Then there is an embedding
  $\varphi':S_g^1\to W_n$ with the following properties.
  \begin{enumerate}[i)]
  \item $\varphi'$ is homotopic to $\varphi$.    
  \item $\varphi'$ is in minimal position with respect to $\Sigma'$.
  \end{enumerate}
\end{lemma}
\begin{proof}
  Note first that removing a sphere $\sigma$ from $\Sigma$ preserves
  minimal position. Hence, we only need to consider the case 
  that $\Sigma' = \Sigma \cup \{\sigma'\}$.

  By assumption, $\varphi$ is in minimal position with respect to
  $\Sigma$. We will therefore only work in $W_\Sigma$ and consider
  $\sigma'$ as a fixed embedded sphere in $W_\Sigma$.
  To put $\varphi$ in minimal position with respect to $\Sigma'$, we  
  need to control how the polygonal disks of $\varphi$ in $W_\Sigma$
  intersect $\sigma'$. 

  To get started, we modify $\varphi$ so that the polygonal disks of
  $\varphi$ in $W_\Sigma$ have a particularly convenient form.
  Namely, let $P_1,\ldots,P_k$ be these polygonal disks (i.e. 
  the components of $\varphi(S_g^1) \cap W_\Sigma$) and let $\Gamma_i
  \subset P_i$ be the embedded decorated ribbon graphs described above.

  By Lemma~\ref{lem:ribbon-control}~i), we may modify $\varphi$ by a
  homotopy, such that each $P_i$ is the surface associated to the
  decorated ribbon graph $\Gamma_i$ without changing
  $\varphi^{-1}(\Sigma)$. 

  We may also assume that after this homotopy $P_i$ is contained in a
  small regular neighborhood of $\Gamma_i$. Intuitively, each $P_i$
  now looks as depicted in Figure~\ref{fig:embedded-graph}.

  \medskip
  As the next step, we apply an isotopy supported in $W_\Sigma$ to $\varphi$
  which puts $P_i$ and $\Gamma_i$ in general position with respect to
  $\sigma'$. More specifically, we can ensure:
  \begin{enumerate}
  \item In a neighborhood of the boundary of $W_\Sigma$, $\varphi$ is
    unchanged. 
  \item $\sigma'$ intersects each $\Gamma_i$ transversely in finitely
    many points.
  \item No intersection point of $\sigma'$ with $\Gamma_i$ is a vertex
    of $\Gamma_i$. 
  \item The intersection between $P_i$ and $\sigma'$ consists of a
    disjoint union of arcs. Each of these arcs corresponds to an
    intersection point of $\Gamma_i$ with $\sigma'$.
  \end{enumerate}

  As a result of this isotopy, each component of $\varphi(S_g^1) \cap
  W_{\Sigma'}$ is a disk whose boundary contains a subarc of $\varphi(\beta)$ 
  and hence $\varphi$ is in ribbon position with respect to
  $\Sigma\cup\{\sigma'\}$. 

  \medskip
  To complete the proof, it remains to show that $\varphi$ can be
  changed by a homotopy to put it in minimal position with respect to
  $\Sigma'$. This will be done inductively, reducing the number of intersections
  of $\varphi(\beta)$ with $\sigma'$. 

  To construct this homotopy of $\varphi$ we will use
  Lemma~\ref{lem:ribbon-control} 
  to homotope the polygonal disks $P_i$ in $W_\Sigma$ relative to
  their boundary arcs. Informally speaking, we will make the graphs
  $\Gamma_i$ (which are one-dimensional objects, and therefore
  completely flexible) intersect $\sigma'$ minimally, and then make
  the $P_i$ follow along. We now give the formal details of 
  this argument.

  \medskip
  Let $b$ be a $\Sigma'$--arc of $\varphi(\beta)$. 
  Assume first that $b$ also is a $\Sigma$--arc. 
  Then $b$ has both endpoints on a sphere distinct from 
  $\sigma'$. By assumption on $\Sigma$, the arc $b$ does not 
  connect the same boundary
  component of $W_{\Sigma}$ to itself. This then also
  holds true for $b$ viewed as a $\Sigma^\prime$-arc. 
 
  If $b$ is not of this form, at least one of its endpoints is
  contained in the sphere $\sigma'$. Suppose that both endpoints of
  $b$ are contained on the same side of $\sigma'$ (alternatively, on
  the same boundary component of $W_{\Sigma'}$). We call such subarcs
  of $\varphi(\beta)$ \emph{problematic}. Note that a problematic subarc is
  disjoint from $\Sigma$. 

  The condition that $\varphi(\beta)$ is in minimal position with
  respect to $\Sigma'$ exactly means that there are no problematic
  arcs. We therefore aim to modify $\varphi$ by a homotopy that
  eliminates all problematic arcs. We will do so inductively, reducing
  the number of problematic arcs. During this induction, we will
  assume that the disks $P_i$ are surfaces defined by decorated ribbon
  graphs $\Gamma_i$. Initially, this is the case, as explained
  above.

\medskip
  Let $P_i$ be the component of $\varphi(S_g^1)\cap W_\Sigma$
  containing $b$ in its boundary. Since $P_i$ is assumed to
  be the surface associated to the decorated ribbon graph $\Gamma_i$,
  the arc $b$ also defines a path $\gamma_b$ in the graph
  $\Gamma_i$. We distinguish two cases.

  First, assume that $\gamma_b$ does not intersect any vertices of $\Gamma_i$.
  Then there is an edge $e$ of $\Gamma_i$ which itself joins the same side of
  $\sigma'$ to itself. Since $W_\Sigma-\sigma'$ is simply connected, there is a
  homotopy of $\Gamma_i$ which moves $e$ to the other side of $\sigma'$.
  By Lemma~\ref{lem:ribbon-control} one can then also homotope $P_i$ 
  (and therefore $\varphi$) such that the number of problematic arcs decreases.
  
  The other case is that $\gamma_b$ intersects at least one vertex of
  $\Gamma_i$. Note that then 
  these vertices are interior vertices of $\Gamma_i$. This case is depicted
  in Figure~\ref{fig:ribbon-to-minimal}.

  Applying collapse moves to $\Gamma_i$ if necessary, we may assume that 
  the arc $\gamma_b$ intersects a single interior vertex $v$ of $\Gamma_i$.
  Furthermore, applying a split move if necessary, we may assume that
  this vertex is trivalent. Namely,  
  there are two edges $e_1,e_2$ meeting at $v$, such that $\gamma_b$
  is a subarc of $e_1\cup e_2$.
  If $v$ is more than trivalent, we apply a split move to generate a new vertex $v'$
  which is adjacent to $e_1, e_2$ and a new edge $e$. Now, $\gamma_b$
  only passes through the trivalent vertex $v'$.

  Since each complementary component of $\sigma'$ in $W_\Sigma$ is
  simply connected, we can 
  homotope the trivalent vertex and $e_1, e_2$ to the other side of $\sigma'$.
  This procedure reduces the number of intersections of $\Gamma_i$ with $\sigma'$.
  Again using Lemma~\ref{lem:ribbon-control}, there is then a homotopy
  of $P_i$ which reduces the number of problematic arcs.

  \begin{figure}[h]
    \centering
    \def\svgwidth{0.7\textwidth}
    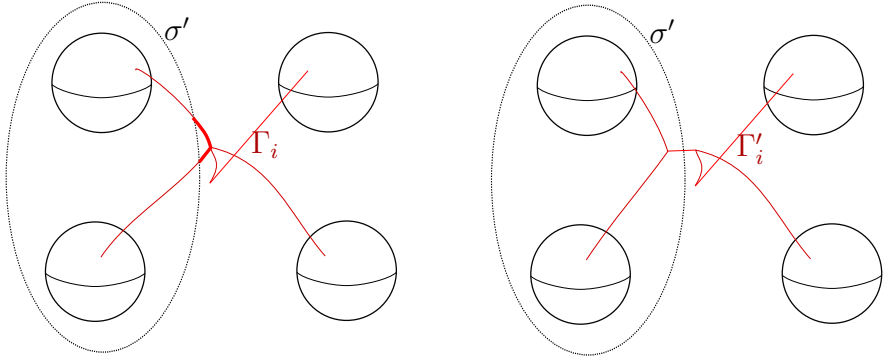  
    \caption{Reducing the number of problematic arcs. On the left: the
    situation before homotopy. The bold part of the graph
    corresponds to $\gamma_b$ for a problematic arc $b$ connecting the
    (left) side of the 
  sphere $\sigma'$ to itself. On the right: the modified graph removing
  this problematic arc.}
    \label{fig:ribbon-to-minimal}
  \end{figure}
  At the end of the induction, there are no problematic arcs, and
  therefore $\varphi$ is in minimal position.
\end{proof}
\begin{cor}\label{cor:minimal-position-exists}
  For every sphere system $\Sigma$, the map $\varphi_0$ may be
  homotoped so that it is in minimal position with respect to $\Sigma$.
\end{cor}
\begin{proof}
  We begin by constructing a sphere system $\Sigma_0$ in the following
  way. Let $A_0$ be a binding arc system of the surface $S_g^1$ with a
  single complementary component. As
  described in Section~\ref{sec:bundles}, the product
  $A_0\times[0,1]$ in the handlebody $U_{2g}$ doubles to a sphere
  system in $W_n$ which we denote by $\Sigma_0$. By construction,
  $\varphi_0$ is in minimal position with respect to
  $\Sigma_0$. Indeed, there is a single polygonal disk corresponding
  to the complementary component of $A_0$.

  Now the corollary follows by induction on the distance of $\Sigma$
  to $\Sigma_0$, using Lemma~\ref{lem:keeping-minimal} and
  Lemma~\ref{lem:inclusion-exclusion}. 
\end{proof}
Next, we strengthen the notion of minimal position further in order to
make the induced arc system unique.
In order to do so, we fix a number of base data once and for all:
\begin{enumerate}
\item An orientation of the boundary curve $\beta\subset S_g^1$.
\item A maximal binding arc system $A_0$ of $S_g^1$.
\item An orientation for each $a\in A_0$.
\end{enumerate}
The additional requirement that will make the position of $\varphi$
unique concerns the position of the arcs $\varphi(a)$ for $a\in
A_0$. Recall from Lemma~\ref{lem:minpos-arcs} that an arc whose
endpoints are allowed to slide on a curve has a well-defined unique minimal
position with respect to a sphere system $\Sigma$ \emph{if it
  intersects $\Sigma$}. If the arc can be made disjoint from $\Sigma$
by a homotopy of arcs relative to the curve, then the complementary
component it is contained in is not well-defined (see
Figure~\ref{fig:minimal-arc}).
  \begin{figure}[h]
    \centering
    \def\svgwidth{0.7\textwidth}
    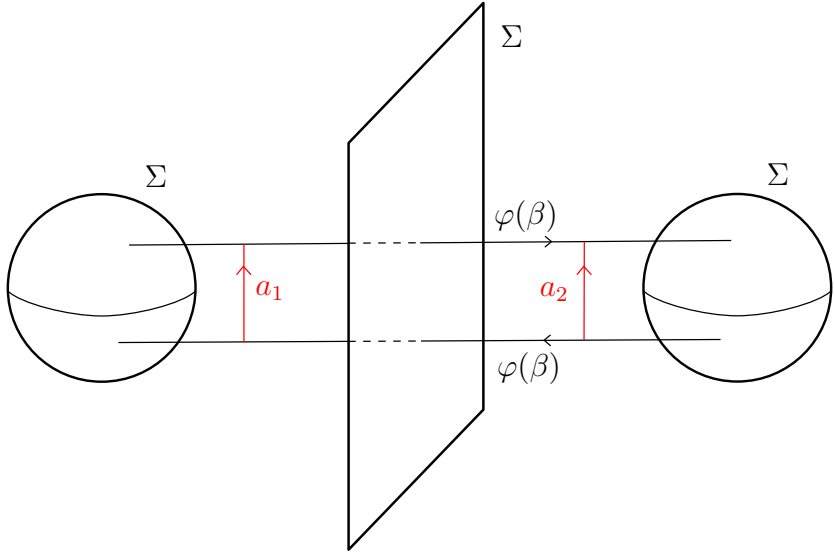  
    \caption{Disjoint arcs do not have a unique minimal position. The arcs
    $a_1$ and $a_2$ are homotopic as arcs relative to
    $\varphi(\beta)$, but are contained in different complementary
    components of $\Sigma$.}
    \label{fig:minimal-arc}
  \end{figure}
To give these arcs a unique position, we use the orientations of the
arcs: we will consider the arc $a_1$ in Figure~\ref{fig:minimal-arc}
preferable to $a_2$, since its initial point is further along
$\varphi(\beta)$ (with its orientation).
\begin{defi}
  We say that $\varphi$ is in \emph{($A_0$-)tight minimal position} if
  it is in minimal position with respect to $\Sigma$, and if
  additionally the following hold.
  \begin{enumerate}[i)]
  \item Each $\varphi(a), a\in A_0$, is in minimal position with
    respect to $\Sigma$ as an arc relative to $\varphi(\beta)$.
  \item If $\varphi(a)$ is disjoint from $\Sigma$, then the initial
    point of $\varphi(a)$ is at the furthest position along
    $\varphi(\beta)$ among all arcs still satisfying i).
  \end{enumerate}
\end{defi}
Note that condition ii) makes sense since $S_g^1$ is not an annulus,
and therefore one cannot push the arc along the boundary indefinitely. We also
remark that condition ii) is not really necessary to make the
arguments work, but just helps to avoid some case distinctions.

We will now show that ($A_0$-)tight minimal position exists and is well-defined.
\begin{lemma}\label{lem:existence-tight-pos}
  For each sphere system $\Sigma$, $\varphi$ may be homotoped to be in
  ($A_0$-)tight minimal position with respect to $\Sigma$.
\end{lemma}
\begin{proof}
  By Corollary~\ref{cor:minimal-position-exists}, we may assume that
  $\varphi$ is in minimal position with respect to $\Sigma$.
  To clear up notation, we will refer to ($A_0$-)tight minimal
  position simply as tight minimal position in this proof and consider
  $A_0$ as fixed.

  We will now improve $\varphi$ to be in tight minimal
  position. 
  We first focus on property~i) of tight minimal position, and
  inductively show that it can be ensured. 
  The idea is that the homotopies needed to put
  $\varphi(A_0)$ in minimal position with respect to $\Sigma$ can be
  done by homotopies of $\varphi$. The main issue is to guarantee that
  $\varphi$ stays an embedding. This requires again the notions of
  polygonal disks and decorated ribbon graphs as in the proof of
  Lemma~\ref{lem:keeping-minimal}. 

  Formally, we will induct on the number of intersection points
  between $\Sigma$ and $\varphi(A_0)$.

  \medskip
  Suppose that $\varphi(A_0)$ is not in minimal position. Then there
  is an arc $a \in A_0$ such that $\varphi(a)$ violates one of the two
  conditions of minimal position of arcs
  (Definition~\ref{def:minimal-position-arc}). 

  \medskip
  Assume that $\varphi(a)$ violates condition~i) of
  Definition~\ref{def:minimal-position-arc}. Then there is a lift
  $\widetilde{a}$ of $\varphi(a)$ to the universal cover
  $\widetilde{W}_n$ with the following property: there is a lift
  $\widetilde{\sigma}$ of a sphere in $\Sigma$ and a subarc
  $\widetilde{b}\subset \widetilde{a}$ which begins and ends on the same side of
  $\widetilde{\sigma}$ and does not intersect any other sphere in
  $\widetilde{\Sigma}$.  

  We now consider the image $b$ of $\widetilde{b}$ in $W_\Sigma$. This
  is an arc connecting one of the boundary components of $W_\Sigma$ to
  itself. Let $P$ be the polygonal disk containing $b$, and let
  $\Gamma$ be the decorated ribbon graph defining $P$.

  We distinguish two cases.
  First suppose that $b$ connects the same
  boundary arc of $P$ to itself. In that case, there is an interval
  $I$ in that boundary arc, such that $I\cup b$ bounds a disk $D$ in
  $P$. In this situation, there is an isotopy of $\varphi$ which does 
  not change the image $\varphi(S_g^1)$, and such that after this 
  isotopy $\varphi^{-1}(b)$ is mapped to $I$. Then one can apply a further 
  isotopy moving $I$ to the side of $\sigma$ not containing $b$
  (in total, one pushes the disk $D$ to the other side of the 
  boundary component; compare the left side of Figure~\ref{fig:slide-arc} 
  for this situation). These isotopies reduce the number of intersections 
  of $\varphi(A_0)$ with $\Sigma$. We may thus assume by induction that 
  we have removed each subarc $b$ of this form. 

  The second case is that $b$ connects two different boundary
  arcs of $P$ to each other. As in the proof of
  Lemma~\ref{lem:keeping-minimal}, such an arc then defines a path
  $\gamma_b$ in $\Gamma$ which contains the central vertex. More precisely,
  $\gamma_b$ consists of two edges $e_1,e_2$ of $\Gamma$. 
  Note that since $\varphi$ is in minimal position, the edges $e_1,e_2$ are 
  not adjacent in the cyclic order given by the ribbon structure, since both of
  them have an endpoint on the same boundary component of $W_\Sigma$.
  Thus, the graph $\Gamma$ has edges both to the right and to the 
  left of $\gamma_b$.

  As a first step to modify $\varphi$, we will perform at most 
  two split moves to ensure that $\gamma_b$ intersects exactly one
  vertex $v$ which then has valence $4$. Namely, if there is more
  than one edge to the right of $\gamma_b$, perform a split separating
  these off to a new vertex. We do the same for the left side (this
  procedure is depicted in the left and middle of
  Figure~\ref{fig:whitehead}). 
 \begin{figure}
    \centering
    \def\svgwidth{\textwidth}
    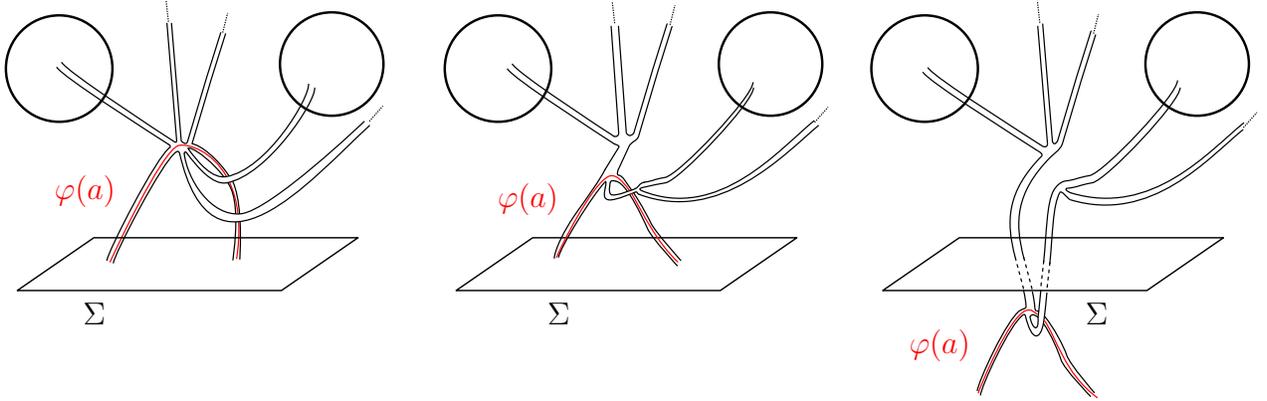  
    \caption{Moving saddles. All three surfaces are in minimal
      position, but the intersection pattern changes.}
    \label{fig:whitehead}
  \end{figure} 

  Furthermore, the following is true: let $b'$ be any component
  of $\varphi(A_0) \cap P$ and let $\gamma_{b'}$ be the arc in $\Gamma$
  it defines. Then $\gamma_{b'}$ shares an endpoint with $\gamma_b$ if
  it intersects $v$. This is due to the fact that $A_0$ is embedded in $S_g^1$
  and therefore $\gamma_{b'}$ does not intersect both an edge to the
  left and to the right of $\gamma_b$.

  Now we can modify $\varphi$ by pushing the vertex $v$ to the other
  side of $\sigma$ (as before, homotoping the graph first, and then
  using Lemma~\ref{lem:ribbon-control} to extend that homotopy to one
  of $\varphi$). As a result of this procedure, the image of the
  arc $b$ does not intersect $\sigma$ anymore (compare the middle and
  right of Figure~\ref{fig:whitehead}).

  To finish the inductive step, we have to show that no new intersections
  of $\varphi(A_0)$ with $\Sigma$ have been generated during this
  homotopy. The only possibility for new intersection points is on the
  image of $P$, and then they would need to be due to arcs $b'$ whose
  corresponding $\gamma_{b'}$ intersects $v$. 
  But, as remarked above, any such $\gamma_{b'}$
  shares an endpoint with $\gamma_b$. Thus, either $b'$ does not
  intersect $\sigma$ after the homotopy (if it was equal to
  $\gamma_b$), or $b'$ intersects $\sigma$ in a single point both
  before and after the homotopy.

  \medskip
  The case of a subarc $b$ which violates condition~ii) of
  Definition~\ref{def:minimal-position-arc} can be handled in the same
  way: either it and a segment in $\varphi(\beta)$ bound a disk in
  some polygonal disk, or it defines a path $\gamma_b$ in some
  $\Gamma$ which connects a boundary vertex to the central
  vertex. As above, one can then apply split moves and push the
  resulting vertex to the other side of $\sigma$.
  By induction, we may therefore assume that $\varphi(A_0)$ is in
  minimal position with respect to $\Sigma$.

  \medskip
  Condition~ii) of tight minimal position can then simply be
  guaranteed by moving those arcs in $\varphi(A_0)$ which are disjoint
  from $\Sigma_0$ along $\varphi(S_g^1)$ to be in the correct position.
\end{proof}

\subsection{Uniqueness of position}
We now show the uniqueness statement for tight minimal position which
is the central tool in the proof of Theorem~\ref{thm:map-in-out}. 

\begin{lemma}\label{lem:uniqueness-tight-position}
  Suppose that $\varphi$ is in ($A_0$-)tight minimal position with respect to
  $\Sigma$. Then the homotopy class of the  binding arc system
  $\varphi^{-1}(\Sigma)$ 
  is determined by the homotopy classes of $\Sigma$ and $\varphi$.
\end{lemma}
In the situation described in the lemma we call $\varphi^{-1}(\Sigma)$
the \emph{induced arc system}.
\begin{proof}
  To clear up notation, we will refer to $(A_0)$-tight minimal
  position simply as tight minimal position in this proof, considering
  $A_0$ as fixed.
  By Lemma~\ref{lem:existence-uniqueness-minimal-intersection-curves},
  the homotopy classes of $\Sigma$ and $\varphi$ determine a unique
  minimal position of $\varphi(\beta)$.
  By Lemma~\ref{lem:minpos-arcs}, all $\varphi(a)$ ($a\in A_0$) which intersect
  $\Sigma$ also have a unique minimal position (in the sense of which
  complementary components of $\widetilde{W}_n$ they cross). By condition ii) of
  tight minimal position the same is true for those $\varphi(a)$ which
  are disjoint from $\Sigma$.

\medskip
  These uniqueness statements do not yet suffice to immediately show that
  $\varphi^{-1}(\Sigma)$ is also determined. The missing piece of data
  is which intersection points of $\varphi(\beta)$
  with $\Sigma$ are joined by arcs in $\varphi(S_g^1)$.

  The strategy of the proof is therefore first to show that (assuming tight
  minimal position) this matching of intersection points is determined
  by the homotopy class of $\Sigma$. This involves
  the study of how images of the complementary pieces of $A_0$ can
  lift to the universal cover of $W_n$.
  Then, one can use the uniqueness of the minimal position of
  $\varphi(\beta)$ to reconstruct the arc system
  $\varphi^{-1}(\Sigma)$ out of the isotopy 
  class of $\Sigma$. We now give the formal details.

\medskip
  A \emph{hexagon disk} $H$ is the image under $\varphi$ of a disk bounded
  by three segments of $\beta$ and three arcs in $A_0$. Since
  $\varphi$ is an embedding, $\varphi(S_g^1)$ is a union of hexagon
  disks which only intersect in their boundaries. In other words,
  $\varphi(S_g^1)\setminus(\varphi(A_0)\cup\varphi(\beta))$ is a
  disjoint union of the interiors of the hexagon disks.
  We first show that every lift $\widetilde{H}$ of a hexagon disk to
  the universal cover $\widetilde{W}_n$ intersects each
  lift $\widetilde{\sigma}$ of a sphere in $\Sigma$ in at most one interval.

  Namely, suppose not. Let $p_1,p_2,p_3,p_4$ be four intersection
  points of the boundary of $\widetilde{H}$ with
  $\widetilde{\sigma}$. For each $i=1,\ldots,4$ let $c_i$ be a lift of the
  corresponding arc in $\varphi(A_0)$ or curve $\varphi(\beta)$
  through $p_i$. 
  Note that since both the arc system $\varphi(A_0)$ and the curve
  $\varphi(\beta)$ are in minimal position, all four of these lifts
  are distinct (each lift may intersect the sphere
  $\widetilde{\sigma}$ at most once).

  Actually, no two of the $c_i$ may intersect in the universal cover
  $\widetilde{W}_n$. Namely, suppose that $c_1$ and $c_2$ intersect.
  Since both $\varphi(\beta)$ and $\varphi(A_0)$ are embedded,
  this means (up to relabeling) that $c_1$ is a lift of
  $\varphi(\beta)$ and $c_2$ is a lift of a component $a$ of $\varphi(A_0)$.
  Let $q$ be the intersection point of $c_1$ and $c_2$.

  Then there is a subarc of $c_1 \cup c_2$ which connects
  $\widetilde{\sigma}$ to itself. This contradicts minimal position of
  $a$ with respect to $\Sigma$.

  Thus, the four $c_i$ are disjoint. The hexagon disk $H$ lifts
  homeomorphically to a disk $\widetilde{H}$ in $\widetilde{W}_n$
  whose boundary is the union of six intervals, four of which would be
  disjoint (as they are contained in the different $c_i$). This is
  clearly impossible. 

\medskip  
  Thus, the lift $\widetilde{H}$ of a hexagon disk
  $H$ intersects the lift $\widetilde{\sigma}$ (in a single arc) if
  and only if one (hence two) of its sides do.

  Note that this condition depends only on the homotopy classes of $\Sigma$
  and $\varphi$ due to the uniqueness of minimal position for curves
  and arcs.

  Hence, the homotopy classes of $\varphi$ and $\Sigma$ determine
  which intersection points of $\varphi(\beta)$ with $\Sigma$ are
  joined by an interval in $\varphi(S_g^1)$ -- namely exactly those
  for which there is a sequence of hexagon disks connecting them.
  This is the desired uniqueness of how the intersection points of
  $\Sigma$ with $\varphi(\beta)$ are matched.

  \medskip
  Let now $x$ and $y$ be two such points which are connected on the
  sphere $\sigma$ by an arc in $\varphi(S_g^1)$.
  Denote by $\beta^1$ and $\beta^2$ the two subarcs of
  $\varphi(\beta)$ defined by these 
  intersection points. The homotopy classes of these subarcs (with
  endpoints sliding on the corresponding sphere of $\Sigma$) are
  completely determined by the sequence of spheres in $\Sigma$ they
  intersect, and thus they are determined by the isotopy class of
  $\Sigma$ and the homotopy class of $\varphi$.

  Let $a\subset S_g^1$ be the preimage of the arc on $\sigma$
  connecting $x$ and $y$.
  The boundary of a regular neighborhood of $\beta\cup a$ in $S_{g}^1$  
  is the union of two simple closed curves
  $d^1,d^2$ and the boundary curve
  $\beta$.

  Up to exchanging $d^1$ and $d^2$, 
  the curve $\varphi(d^k)\subset W_n$ is freely homotopic 
  to a curve $\delta^k = \beta^k * \alpha$ obtained by concatenating
  $\beta^k$ and an embedded arc $\alpha$ on $\sigma$.

  Since $\sigma$ is simply connected, the free homotopy classes of
  the curves $\delta^k$ are thus also determined by the isotopy class of
  $\Sigma$ and the homotopy class of $\varphi$.

  Since $\varphi$ induces an isomorphism on the
  level of fundamental groups, this implies that also 
  the simple closed curves $d^k$ are determined by this data.

  The curves $\beta$, $d^1$ and $d^2$ bound a pair of pants $P$ on
  $S_g^1$. The arc $a$ is up to isotopy the unique essential embedded
  arc in $P$ connecting $\beta$ to itself.
  Thus the isotopy class of the arc $a$ is determined by the isotopy class of
  $\Sigma$ and the homotopy class of $\varphi$. 

  Since this argument applies to all arcs
  $a\subset\varphi^{-1}(\Sigma)$, this proves the desired uniqueness.
\end{proof}

We now have collected all the necessary tools to prove the main
theorem of this section.
\begin{proof}[Proof of Theorem~\ref{thm:map-in-out}]
  To prove the theorem, we fix a maximal binding arc system $A_0$, and
  orientations of the boundary $\beta$ as well as each arc $a\in A_0$
  as before. Let $\Sigma_0$ be the sphere system obtained by doubling
  the arc system $A_0$ (as described in the last paragraph of
  Section~\ref{sec:bundles}) 

  We define a $1$--Lipschitz projection $P$ of the graph of simple sphere
  systems to the graph of binding arc systems as follows.

  For a simple sphere system $\Sigma$, modify $\varphi_0$ by a homotopy
  to a map $\varphi$ in $(A_0)$-tight minimal position, and let $P(\Sigma) =
  \varphi^{-1}(\Sigma)$ be the induced 
  binding arc system. By Lemma~\ref{lem:uniqueness-tight-position} the
  result is determined by the homotopy class of $\Sigma$.

  This map is $1$-Lipschitz since disjoint sphere systems are mapped
  to disjoint binding arc systems. Namely, apply
  Lemma~\ref{lem:existence-tight-pos} to the union of two
  disjoint sphere systems to see that there is a simultaneous
  tight position for both of them.

  Let now $f\in\Map(S_{g,1})$ be given and let $\Sigma=\iota(f)(\Sigma_0)$.
  By doubling a representative $I(f)$ we find that $f(A_0)$ is the
  intersection of $I(f)(\Sigma)$ and $\varphi$. In particular, it is
  in tight position.
  Thus, $P$ restricts to the identity on the graph of binding arc
  systems. This shows the theorem.
\end{proof}

\subsection{Arc graphs}
\label{sec:arcgraph}
The method employed in the proof of Theorem \ref{thm:map-in-out} can
also be used to relate the arc graph of a punctured surface to the
sphere graph of $W_n$.

To be precise, recall that the \emph{arc graph $\mathcal{AG}(S_g^1)$
  of $S_g^1$} is the graph whose vertex set is the set of isotopy classes
of embedded essential arcs connecting the boundary of $S_g^1$ to
itself. 
Again, isotopies are only required to fix the boundary component setwise.
Two such vertices are joined by an edge if the corresponding
arcs can be embedded disjointly. 
Similarly,
define the \emph{sphere graph $\mathcal{SG}(W_n)$ of $W_n$} to be the graph whose vertex set
is the set of isotopy classes of essential $2$-spheres in $W_n$. Two
such vertices are connected by an edge if the corresponding 
spheres can be realized disjointly.

Let $a$ be an arc representing a vertex of the arc graph of $S_g^1$.
The interval bundle over $a$ is a disk $D(a)$ in the handlebody
$U_{2g}=S_g^1\times [0,1]$. 
The isotopy class of this disk only depends on the isotopy class
of $a$, since the Dehn twist around the boundary of $S_g^1$ is
contained in the kernel of the map $\Map(S_g^1)\to\Map(U_{2g})$.
We let $\sigma(a)$ be the essential sphere in $W_n$ which is obtained by
doubling $D(a)$ along $\partial U_{2g}$.

The following lemma is folklore, but since we were not able
to find a proof in the literature, we include a proof.
\begin{lemma}
  The construction above identifies the arc graph of $S_g^1$ with a
  subgraph of the sphere graph of $W_n$.   
\end{lemma}
\begin{proof}
  The only thing that requires an argument is the injectivity of the
  map. First note that $\sigma(a)$ is nonseparating if and only if $a$
  is nonseparating. 

  Let $a, a'$ be two non-isotopic, nonseparating arcs on
  $S_g^1$. Let $G$ (resp. $G'$) be the corank--$1$ subgroup of $\pi_1(S_{g,1},p)$
  of those loops which are disjoint from $a$ (resp. $a'$). Since $a$
  and $a'$ are non-isotopic, we claim that these groups are not
  equal. Namely, if a 
  loop can be made disjoint from $a$ and $a'$ individually, then it
  can also be made disjoint from $a \cup a'$. Since the rank of the
  fundamental group of the complement decreases strictly when adding
  more arcs, the claim follows.

  Since the map $\varphi_0$ induces an isomorphism on $\pi_1$, the image groups
  $(\varphi_0)_*(G), (\varphi_0)_*(G')$ are therefore also different corank--$1$ subgroups.
  The subgroup of loops in $W_n$ which can be made disjoint from $\sigma(a)$ 
  (resp. $\sigma(a')$) is a corank--$1$ subgroup which contains $(\varphi_0)_*(G)$, and
  is therefore equal to $(\varphi_0)_*(G)$ (resp. $(\varphi_0)_*(G')$).

  Hence, $\sigma(a)$ and $\sigma(a')$ are not isotopic, since that would
  imply $G=G'$.

  The case of non-isotopic separating arcs $a,a'$ can be proved
  similarly by considering both complementary components
  simultaneously (instead of a corank--$1$ subgroup one then considers
  a free splitting of the fundamental group).
\end{proof}

\begin{prop}\label{prop:arcgraph}
  The arc graph of $S_g^1$ is a 1-Lipschitz retract of the sphere
  graph of $W_n$. In particular, it is undistorted.
\end{prop}

\begin{proof}[Proof of Proposition~\ref{prop:arcgraph}]
  We define the Lipschitz retraction in a similar way as in the proof
  of Theorem~\ref{thm:map-in-out}. 

  Let $\sigma$ be an essential sphere in $W_n$. Extend
  $\sigma$ to a simple sphere system $\Sigma$. Put $\Sigma$ in tight
  minimal position with respect to $\varphi$.
  Let $a(\sigma) \subset P(\Sigma)$ be the part of the induced arc
  system which is the preimage of $\sigma$. Note that this is a
  nonempty set of essential arcs. Namely, if $a(\Sigma)$ were empty,
  then the full fundamental group of $S_g^1$ would inject in the
  fundamental group of the complement of $\sigma$, which is impossible
  since $\sigma$ is essential.

  We claim that $a(\sigma)$ does not depend on the choice of
  $\Sigma$. Any two possible choices $\Sigma,\Sigma'$ of extensions
  differ by a sequence of moves, each of which adds or removes a
  sphere different from $\sigma$. This is an immediate consequence of
  the fact that the graph $\mathcal{S}(W_n,\sigma)$ defined in
  Section~\ref{sec:spheres} is connected.

  Now, arguing as in the proof of Theorem~\ref{thm:map-in-out}, each
  such move does not change the preimage of $\sigma$ in $P(\Sigma)$.
  Thus, $a(\Sigma)$ is a well-defined arc in $S_g^1$, and $a$ defines
  a map from the sphere graph of $W_n$ to the arc graph of $S_g^1$. It
  is clear that this map restricts to the identity on the arc graph.

  If $\sigma_1$ and $\sigma_2$ are two disjoint essential spheres,
  then one can find a simple sphere system $\Sigma$ containing both
  $\sigma_1$ and $\sigma_2$. Thus, $a(\sigma_1)$ and $a(\sigma_2)$ are
  both contained in $P(\Sigma)$ and thus in particular disjoint.
  This shows that $a$ is $1$-Lipschitz.
\end{proof}

\end{document}